\documentclass{amsart}
\usepackage{latexsym}
\usepackage{amsmath}
\usepackage{amsthm}
\usepackage{amssymb}

\newtheorem{thm}{Theorem}
\newtheorem{lem}[thm]{Lemma}

\newtheorem{cor}[thm]{Corollary}
\newtheorem{defi}[thm]{Definition}

\newtheorem{prop}[thm]{Proposition}
\newtheorem{rk}[thm]{Remark}

\newcommand{\vip}{\vskip.2cm}
\newcommand{\e}{{\varepsilon}}
\newcommand{\rr}{{\mathbb{R}}}
\newcommand{\CC}{{\mathcal D}}

\newcommand{\nn}{{\mathbb{N}}}
\newcommand{\rd}{{\mathbb{R}^d}}
\newcommand{\E}{\mathbb{E}}
\newcommand{\Leb}{{\rm Leb}}
\newcommand{\Var}{\mathbb{V}{\rm ar}\,}
\newcommand{\Cov}{\mathbb{C}{\rm ov}}
\newcommand{\bh}{{\bf h}}
\newcommand{\bg}{{\bf g}}
\newcommand{\bff}{{\bf f}}
\newcommand{\bm}{{\bf m}}
\newcommand{\bM}{{\bf M}}
\newcommand{\bR}{{\bf R}}
\newcommand{\ba}{{\bf a}}
\newcommand{\bb}{{\bf b}}
\newcommand{\Exp}{{\rm Exp}}

\newcommand{\cN}{\mathcal{N}}
\newcommand{\indiq}{{{\mathbf 1}}}
\newcommand{\intoi}{\int_0^\infty}
\newcommand{\intrd}{\int_\rd}
\newcommand{\dd}{{\rm d}}

\begin{document}

\title{On the Kozachenko-Leonenko entropy estimator}

\author{Sylvain Delattre}

\author{Nicolas Fournier}

\address{Sylvain Delattre, Laboratoire de Probabilit\'es et Mod\`eles Al\'eatoires, UMR 7599, 
Universit\'e Paris Diderot, Case courrier 7012, avenue de France, 75205 Paris Cedex 13, France.}

\email{sylvain.delattre@univ-paris-diderot.fr}

\address{Nicolas Fournier, Laboratoire de Probabilit\'es et Mod\`eles Al\'eatoires, UMR 7599,
Universit\'e Pierre-et-Marie Curie, Case 188, 4 place Jussieu, F-75252 Paris Cedex 5, France.}

\email{nicolas.fournier@upmc.fr}

\begin{abstract}
We study in details the bias and variance of the
entropy estimator proposed by Kozachenko and Leonenko \cite{kl}
for a large class of densities on $\rd$. We then use the work of Bickel and Breiman \cite{bb}
to prove a central limit theorem in dimensions $1$ and $2$.
In higher dimensions, we provide a development of the bias in terms 
of powers of $N^{-2/d}$. This allows us
to use a Richardson extrapolation to build, in any dimension, an estimator satisfying a 
central limit theorem and for which we can give some 
some explicit (asymptotic) confidence intervals.
\end{abstract}



\maketitle

\section{Introduction and main results}

\subsection{The setting}

Consider a probability measure $f$ on $\rd$. We also denote by $f$ its density.
We are interested in its entropy defined by
$$
H(f)=-\intrd f(x) \log f(x)\dd x.
$$ 
For $N\geq 1$ and for $X_1,\dots,X_{N+1}$ an i.i.d. sample of $f$, we consider, for each $i=1,\dots,N+1$,
\begin{equation}\label{rz}
R^N_i=\min\{|X_i-X_j| \; : \; j=1,\dots,N+1,\; j\ne i\}\quad \hbox{and} \quad
Y^N_i=N(R^N_i)^d.
\end{equation}
Here $|\,\cdot\,|$ stands for any norm on $\rd$. For $x\in\rd$ and $r\geq 0$, we set
$B(x,r)=\{y\in\rd\; :\; |y-x|\leq r\}$ and we introduce 
$v_d=\int_{B(0,1)}\dd x$.
We also denote by $\gamma=-\intoi e^{-x}\log x \dd x\simeq
0.577$ the Euler constant. We finally set
\begin{equation}\label{mh}
H_N= \frac 1 {N+1}\sum_{i=1}^{N+1} \log Y^N_i +\gamma+\log v_d.
\end{equation}
The estimator $H_N$ of $H(f)$ was proposed by Kozachenko and Leonenko
\cite{kl}.
The object of the paper is to study in details the bias, variance
and asymptotic normality of $H_N$.

\subsection{Heuristics}\label{heu}
Let us explain briefly why $H_N$ should be consistent.

\vip

The conditional law of $Y^N_i$ knowing $X_i$ is approximately $\Exp(v_df(X_i))$ for $N$ large:
for $r>0$, ${\Pr(Y^N_i>r \,|\, X_i)= [1-f(B(X_i,(r/N)^{1/d}))]^N \simeq 
\exp(-N f(B(X_i,(r/N)^{1/d})) )
\simeq \exp(-v_df(X_i)r)}$. 

\vip

Consequently, we expect that
$Y^N_i=\xi_i/(v_d f(X_i))$,
for a family $(\xi_i)_{i=1,\dots,N+1}$ of approximately 
$\Exp(1)$-distributed random variables, hopefully not too far from being independent.

\vip

We thus expect that
$(N+1)^{-1} \sum_{i=1}^{N+1} \log Y^N_i
\simeq \E[\log (\xi_1/(v_df(X_1)))]=\E[\log \xi_1]-\log v_d - 
\E[\log f(X_1)]\simeq -\gamma-\log v_d +H(f)$
and thus that $H_N \simeq H(f)$.

\subsection{Motivation}

Let us first reproduce the introductory paragraph of the Wikipedia page {\it Entropy estimation} \cite{w}:
``In various science/engineering applications, such as independent component analysis,
image analysis, genetic analysis, speech recognition, manifold learning, and time delay estimation, 
it is useful to estimate the differential entropy of a system or process, given some observations.''
Starting from this Wikipedia page and following the references therein,
we find numerous applied papers that we do not try to summarize, we also
found many papers dealing with entropy and its estimation
in kinetic physics, molecular chemistry, computational neuroscience, etc.
Beirlant, Dudewicz, Gy\"orfi and van der Meulen \cite{bdgvdm} also mention applications
to quantization, econometrics and spectroscopy.

\vip

The estimation of the relative entropy (or Kullback-Leibler divergence)
of $f$ with respect to some know $g$ is easily deduced from the entropy
estimation, since $H(f\vert g)=-H(f)+\intrd f(x) \log g(x) \dd x$ and since $\intrd f(x) \log g(x)
\dd x$ is easily estimated.

\vip

Concerning applications to statistics, let us mention a few goodness-of-fit tests based on the 
entropy estimation: see Vasicek \cite{v} for normality (Gaussian laws maximize
the entropy among all distributions with given variance),
Dudewicz and van der Meulen \cite{dvdm}
for uniformity (uniform laws maximize the entropy among
all distributions with given support), Mudholkar and Lin \cite{ml} for exponentiality
(exponential laws maximize
the entropy among all $\mathbb{R}_+$-supported distributions with given mean).
Also, Robinson \cite{r} proposed some independence test, based on the fact that
$f\otimes g$ maximizes $H(F)$ among all distributions $F$ with marginals $f$ and $g$.

\subsection{Available results}

We now list a few mathematical results. Let us first mention
the review paper \cite{bdgvdm} by Beirlant, Dudewicz, Gy\"orfi and van der Meulen.

\vip

Levit \cite{l} has shown that $\Var(\log f(X_1))=\intrd f(x)\log^2f(x)\dd x - (H(f))^2$
is the smallest possible normalized (by $N$) asymptotic
quadratic risk for entropy estimators in the local minimax sense.

\vip

Essentially, there are two types of methods for the entropy estimation. The {\it plug in} method
consists in using an estimator of the form $H_N=-\intrd f_N(x)\log f_N(x)\dd x$, where $f_N$ is an estimator
of $f$. One then needs to use something like a kernel density estimator and this
requires to have an idea of the tail behavior of $f$.
Joe \cite{j} considers the case where $f$ is bounded below on its (compact) support,
while Hall and Morton \cite{hm} propose some root $N$ and asymptotically normal estimators
assuming that $f(x)\sim a |x|^{-\alpha}$ (with $\alpha$ known) or
$f(x)\sim a \exp(-b|x|^{-\alpha})$ (with $\alpha$ known).

\vip

The second class of methods consists in using {\it spacings} if $d=1$, see
Vasicek \cite{v}, or {\it neighbors} as proposed by Kozachenko and Leonenko \cite{kl}.
In \cite{kl}, a consistency result is proved (for $H_N$ defined by \eqref{mh}), in any dimension, under rather
weak conditions on $f$ and this is generalized to other notions of entropies
by Leonenko, Pronzato and Savani in \cite{lps}.
Instead of using {\it nearest} neighbor, we can use $k$-th nearest neighbors with
either $k$ fixed or $1<<k<<N$ (similarly, in dimension $1$, we can use $k$-spacings).

\vip

In dimension $1$ and assuming that $f$ is bounded below on its (compact) support, Hall \cite{h,h2} and
van Es \cite{vE} show some root $N$ consistency and asymptotic normality for the entropy estimator
based on $k$ spacings (in both cases where $k$ is fixed or tends to infinity).
Tsybakov and van der Meulen \cite{tvdm} are the first to prove some root $N$ consistency for
some entropy estimator for general densities with unbounded support, in dimension $1$.
They consider a modified version of
\eqref{mh} and assume that $f$
is sufficiently regular, positive and has some sub-exponential tails.
Still in dimension $1$, El Haje and Golubev \cite{eg}
prove some root $N$ consistency and asymptotic normality
for the entropy estimator based on $1$-spacings. Furthermore, their assumptions on $f$ are
weaker than those of \cite{tvdm}. In particular, they allow $f$ to have some zeroes and some fat tails.

\vip

In any dimension, Bickel and Breiman \cite{bb} prove a very general central limit theorem
for estimators, based on nearest neighbors, of nonlinear functionals, 
unfortunately not including the entropy. Also, they do not study the bias.

\vip

Let us finally mention the paper of P\'al, P\'oczos and Szepesv\'ari:
they study other notions of entropy, work in dimension $d\geq 1$, use estimators based on 
nearest neighbors and quantify the consistency.

\subsection{Notation}
Let $r_0>0$ be fixed (we will assume for simplicity that $r_0=1$ in the proofs). 
We introduce the constant $\kappa$ and the functions $m,M:\rd \mapsto (0,\infty)$ defined by
\begin{align}\label{mM}
\kappa=\sup_{x\in\rd} f(B(x,r_0)),\quad 
m(x)=\inf_{\e\in(0,r_0)} \frac{f(B(x,\e))}{\e^d} \quad \hbox{and} \quad M(x)=\sup_{B(x,2r_0)} f.
\end{align}
Observe that we have $m\leq v_d f\leq v_d M$ (at least if $f$ is continuous).

\vip

For $\beta>0$, we say that $f\in \CC^\beta(\rd)$ if $f \in C^k(\rd)$ for
$k=\max\{i \in \nn \; : \; i<\beta\}$ and if $D^k f$ is locally 
H\"older continuous with index $\beta-k$. We then set (with $\sum_1^0=0$ when $k=0$)
\begin{equation}\label{gbeta}
G_\beta(x)=\sum_{i=1}^k |D^i f(x)| + \sup_{y\in B(x,r_0)}\frac{|D^kf(y)-D^kf(x)|}{|x-y|^{\beta-k}}.
\end{equation}

\subsection{Important convention} We write $\intrd \phi(x) f(x)\dd x$ for $\E[\phi(X_1)]=
\int_{\{f>0\}}\phi(x) f(x)\dd x$,
even if $\phi$ is not well-defined on $\{f=0\}$.

\subsection{Raw results}
We give here our results in a gross way: we write exactly what we can prove.
The next section contains some more comprehensible corollaries.
We start with the asymptotic normality and variance.

\begin{thm}\label{thv}
Assume that $f$ is bounded and continuous and that, for some $q>0$, some $\theta \in (0,1)$,
\begin{equation}\label{condv}
\intrd \Big(|x|^q +\log^2 m(x)+\frac{\log^2(2+|x|)}{(f(B(x,r_0)))^\theta} 
+\frac {M(x)}{m(x)}(1+|\log m(x)|)\Big)
f(x)\dd x<\infty.
\end{equation}

(i) If $\theta>1/2$, then $\sqrt N (H_N-\E[H_N])$ goes in law to $\cN(0,\sigma^2(f))$ as $N\to \infty$,
where 
$$
\sigma^2(f)=\intrd f(x)\log^2f(x)\dd x - (H(f))^2+\chi_d,$$ 
with
$$
\chi_d=2\log 2 + \frac{\pi^2}{6} - 1 + \intoi \intoi e^{-u-v}T(u,v) \frac{\dd u}{u}\frac{\dd v}{v},
$$
and $T\!:\!(0,\infty)^2\mapsto (0,\infty)$ defined by $T(v_dr^d,v_ds^d)={\int_{B(0,r+s)\setminus B(0,r\lor s)}
[\exp(\int_{B(0,r)\cap B(y,s)}\dd z) -1] \dd y}$.

\vip

(ii) If $\theta \in (0,1/2]$, then there is a constant $C$ such that  for all $N\geq 1$,
$\Var H_N \leq C N^{-2\theta}$.
\end{thm}

\begin{rk} 
We recall that $v_d=\pi^{d/2}/\Gamma(d/2+1)$ for $|\,\cdot\,|_2$ and $v_d=2^d$ for $|\,\cdot\,|_\infty$.
The constant $\chi_d$ depends on $d$ and on the norm $|\,\cdot\,|$.  We computed it numerically
by a Monte-Carlo method in the following situations:
\begin{center}
\begin{tabular}{|c||c|c|c|c|c|c|c|c|c|c|c|c|c|}
\hline
&d=1&d=2&d=3&d=4&d=5&d=6&d=7&d=8&d=9&d=10&d=20\\
\hline
$\chi_d$ for $|\,\cdot\,|_2$ & 2.14 & 2.29 & 2.42 & 2.52 & 2.61 & 2.67 & 2.7 & 2.7 & 2.8 & 2.9 & ? \\
\hline
$\chi_d$ for $|\,\cdot\,|_\infty$ & 2.14 & 2.31 & 2.47 & 2.60 & 2.70 & 2.78 & 2.84 & 2.88 & 2.91 & 2.94 & 3.03 \\
\hline
\end{tabular}
\end{center}
The values for $|\,\cdot\,|_\infty$ are more reliable because our Monte-Carlo method is unbiased
(this is due to the fact that we can compute explicitly the volume of the intersection of two balls).
We did not indicate any approximate value of $\chi_{20}$ for $|\,\cdot\,|_2$ because we obtain too variable results.
\end{rk}

We next study the bias.

\begin{thm}\label{thb1}
Let $\beta \in (0,2]\cap(0,d]$ and assume that
$f \in \CC^\beta(\rd)$. Put $R=M+G_\beta$.
Assume that $\kappa<1$ and that there are $q>0$ and $\theta \in (0,\beta/d]$ such that
\begin{equation}\label{condb1}
\intrd \Big(|x|^q+ \frac{\log(2+|x|)}{(f(B(x,r_0)))^\theta}+\frac{R^{2\theta}(x)}{m^{2\theta}(x)}
+\frac{R^{\theta d/\beta}(x)}{m^{\theta(d+\beta)/\beta}(x)} \Big) f(x)\dd x<\infty.
\end{equation}
Then there is a constant $C$ such that for all $N\geq 1$,
\begin{align*}
 |\E[H_N]-H(f)|\leq \frac C {N^{\theta}}.
\end{align*}
\end{thm}

This provides at best a bias in $O(N^{-2/d})$
and this is a natural limitation: as seen in Subsection \ref{heu}, 
we need to approximate $v_df(x)r$ by $Nf(B(x,(r/N)^{1/d}))$, for which the error is of order $N^{-2/d}$.
In dimensions $1,2,3$, $N^{-2/d}=o(N^{-1/2})$ and we can hope that the bias is negligible
when compared to the standard deviation. But in higher dimension, it will be predominant. 
To get a smaller bias, 
one possibility is to use a Richardson extrapolation, which requires a polynomial development.

\begin{thm}\label{thb2}
Assume that $d\geq 3$, let $\beta \in (2,d]$ and set
$\ell=\max\{i\in\nn \; : \; 2i<\beta\}\geq 1$.
Assume that $f \in \CC^\beta(\rd)$. Put $R=M+G_\beta$.
Assume that $\kappa<1$ and that there is $q>0$ such that
\begin{equation}\label{condb2}
\intrd \Big( |x|^q +  \frac{\log(2+|x|)}{(f(B(x,r_0)))^{\beta/d}}+\frac{R^{2\beta/d}(x)}{m^{2\beta/d}(x)}
+\frac{R^{\beta/2}(x)}{m^{\beta/d+\beta/2}(x)}\Big)f(x)\dd x<\infty.
\end{equation}
Then there are $\lambda_1,\dots,\lambda_\ell \in \rr$ and $C>0$ such that
for all $N\geq 1$,
\begin{align*}
\Big|\E[H_N]-H(f)-\sum_{i=1}^{\ell} \frac{\lambda_i}{N^{2i/d}} \Big|\leq 
\frac C  {N^{\beta/d}}.
\end{align*}
\end{thm}

\begin{rk}\label{lambda1}
The $\lambda_i$'s can be made explicit. 
In particular, if the norm is symmetric (that is, 
$|(x_{\sigma(1)},\dots,x_{\sigma(d)})|=|(x_1,\dots,x_d)|$ for any permutation $\sigma$),
there is $c>0$, depending only on $d$
and on the norm, such that 
$\lambda_1=c \intrd f^{-2/d-1}(x)|\nabla f(x)|^2\dd x$
provided $\lim_{|x|\to \infty}f^{-2/d}(x)|\nabla f(x)|=0$.
\end{rk}

To produce some confidence intervals, we need to estimate $\sigma^2(f)$, which is not very difficult.

\begin{prop}\label{mp}
Assume that $f$ is bounded and continuous and that for some $q>0$, some $\theta>0$,
\begin{equation}\label{condestiv}
\intrd \Big(|x|^q +\log^2 m(x)+\frac{\log^2(2+|x|)}{(f(B(x,r_0)))^\theta} 
\Big) f(x)\dd x<\infty.
\end{equation}
Fix $N\geq 1$, recall \eqref{rz} and put 
$$
V_N=\frac1{N+1} \sum_{i=1}^{N+1} \log^2 Y^N_i - \Big(\frac1{N+1} \sum_{i=1}^{N+1} \log Y^N_i\Big)^2 
+\chi_d-\frac{\pi^2}6.
$$
Then $V_N$ goes in probability to $\sigma^2(f)$, defined in Theorem \ref{thv}, as $N\to\infty$.
\end{prop}

\subsection{Corollaries}

We now sacrifice generality to present a more comprehensible statement
with a central limit theorem sufficient to produce
explicit (asymptotic) confidence intervals.
Weaker results can be derived from the theorems of the previous section
under weaker assumptions, we give an example at the end of the subsection.

\vip

If $d=1,2,3$, we set $H^{(d)}_N=H_N$ and $a_d=1$.
If $d\geq 4$, it is necessary to proceed to some extrapolation to get a bias in $o(N^{-1/2})$.
We consider $\ell=\lfloor d/4\rfloor$ and the real numbers
$\alpha_{0,d},\dots,\alpha_{\ell,d}$ satisfying 
$\sum_{k=0}^\ell \alpha_{k,d}=1$ and, for all $i=1,\dots,\ell$,
$\sum_{k=0}^\ell \alpha_{k,d} 2^{2ki/d} =0$.
For $N \geq \ell$, we put $n=\lfloor(N+1-\ell)/(2^{\ell+1}-1)\rfloor$, so that
$N+1 \geq \sum_{k=0}^\ell(2^k n +1)$.
We thus can split our 
$(N+1)$-sample into $\ell+1$
(independent) sub-samples of sizes $2^\ell n+1,2^{\ell-1}n+1,\dots,n+1$ and build with
these sub-samples the estimators $H^0_{2^\ell n},H^1_{2^{\ell-1} n},\dots H^\ell_{n}$
exactly as in \eqref{rz}-\eqref{mh}. Finally, we set 
\begin{equation}\label{richard}
H^{(d)}_N=\sum_{k=0}^\ell \alpha_{k,d}H^k_{2^{\ell-k}n}
\end{equation}
and put $a_d=(2-2^{-\ell})\sum_{k=0}^\ell \alpha_{k,d}^2 2^{k}$.

\vip

\begin{cor}\label{c1}
Fix $\e\in(0,1)$. Assume that $f\in \CC^{\nu}(\rd)$ with $\nu=1$ if $d=1$, 
$\nu=2$ if $d\in\{2,3\}$ and $\nu=d/2+\e$ if $d\geq 4$.
Fix $r_0>0$ such that $\kappa<1$, assume that $R=M+G_\nu$ is bounded and
that there is $c>0$ such that $m(x)\geq cf(x)$ for all $x \in \rd$.
Assume finally that 
\begin{align}\label{cc1}
\intrd |x|^{d+\e}f(x)\dd x<\infty
\end{align}
and that
\begin{align}\label{cc2}
\int_{\{f> 0\}}\!\!\!\!\!\!\sqrt {R(x)} f^{-\e}(x) \dd x <\infty \;\; \hbox{if $d\in\{1,2\}$}
\;\;  \hbox{and} \;\; 
\int_{\{f> 0\}} \!\!\!\!\!\!R^{d/4}(x) f^{1/2-d/4-\e}(x)\dd x <\infty \quad \hbox{if $d\geq 3$.}
\end{align}
Then 
$$
\sqrt{\frac N{V_N}}\Big(H_N^{(d)}-H(f)\Big) \longrightarrow \cN(0,a_d) 
\quad \hbox{in law as $N\to\infty$.}
$$
\end{cor}

\begin{rk}
For $d\in\{4,5,6,7\}$, we have $\ell=1$, $\alpha_{0,d}=2^{2/d}/(2^{2/d}-1)$ 
and $\alpha_{1,d}=-1/(2^{2/d}-1)$. This gives $a_4\simeq 34.97$,
$a_5\simeq 54.97$, $a_6\simeq 79.65$, $a_7\simeq 109.01$. 
\end{rk}

These values are very large and it is not clear,
in practice, that the extrapolation is judicious, except if $N$ is very large.
However, this is the only way we have to propose a satisfying {\it theoretical} 
result when $d\geq 4$.

\vip

Let us discuss our assumptions.

\vip

$\bullet$ We assume some regularity on $f$. This is natural, since 
we need to use that $f$ is well-approximated by its means of on small balls.

\vip

$\bullet$ We suppose that $m \geq c f$ for some constant $c>0$.
Although this condition is not automatically verified, it is not so easy
to find a counter-example.

\vip

$\bullet$ We have some decay condition: $f$ needs to have a moment of order strictly larger
than $d$. This is not very stringent, but we believe this is a technical condition.
Observe that when $f$ is sufficiently regularly varying and positive so that 
$f\simeq R\simeq m$, in some  weak sense to be precised
(we think here of Gaussian distributions or of examples (a), (b), (c) below),
then our main conditions are $\intrd |x|^{d+\e}f(x)\dd x<\infty$
and $\intrd f^{1/2-\e}(x) \dd x <\infty$. And the first condition implies the second one 
(with actually a smaller $\e'>0$),
see Remark \ref{supercool}-(i).
To summarize, for a non-vanishing and very regularly varying $f$, our main restriction
is \eqref{cc1}, which is not very stringent but probably technical.

\vip

$\bullet$ When $d\in\{1,2\}$, $f$ is plainly allowed to vanish:
for example, Corollary \ref{c1} applies to any compactly supported $f$,
provided it is of class $C^d(\rd)$ (whence $R$ is bounded), provided $m\geq c f$ and provided
there is $\e>0$ such that $\int_{\{f>0\}} f^{-\e}(x)\dd x<\infty$. This is rather general.
On the contrary, $f$ cannot vanish if $d\geq 10$. Indeed, consider
$x_0$ on the boundary of $\{f>0\}$, so that $R(x_0)>0$. Since $f\in\CC^{d/2}(\rd)$ by assumption,
we deduce that $f(x_0+h)\leq C |h|^{d/2}$ for $h\in\rd$ small enough,
so that $R(x_0+h)f^{1/2-d/4-\e}(x_0+h)\geq c |h|^{d(1/2-d/4-\e)/2}$, which cannot be integrable
around $h=0$ if $(d/4-1/2)/2\geq 1$, i.e. if $d\geq 10$.
For $d\in\{3,\dots,9\}$, $f$ is allowed to vanish with restrictions.
Again, we believe these restrictions are technical.

\vip

Finally, let us state a corollary with a bad rate of convergence
but with very few assumptions.

\begin{cor}\label{c2}
Assume that $f\in \CC^\nu(\rd)$ with $\nu=\min\{d,2\}$.
Fix $r_0>0$ such that $\kappa<1$, assume that $R=M+G_\nu$ is bounded and
that there is $c>0$ such that $m(x)\geq cf(x)$ for all $x \in \rd$.
Assume finally \eqref{cc1} for some $\e>0$, that 
$\int_{\{f> 0\}} M(x)|\log f(x)|\dd x<\infty$ and that 
\begin{align}\label{cc3}
\intrd \sqrt{R(x)}  \dd x <\infty \quad \hbox{if $d=1$}
\quad  \hbox{and} \quad
\intrd R^{d/(2+d)}(x)\dd x <\infty \quad \hbox{if $d\geq 2$.}
\end{align}
Then there is a constant $C>0$ such that for all $N\geq 1$, 
$\E[(H_N-H(f))^2]\leq C N^{-1}$ if $d=1$ and 
$\E[(H_N-H(f))^2]\leq C N^{-4/(d+2)}$ if $d\geq 2$.
\end{cor}

Any $f\in C^{\min\{d,2\}}(\rd)$ with compact support such that $m\geq cf$ and 
$\int_{\{f> 0\}} |\log f(x)|\dd x<\infty$
satisfies the assumptions of Corollary \ref{c2}.
It is of course possible, using Theorems \ref{thv}, \ref{thb1} and \ref{thb2}
to derive intermediate statements between Corollaries \ref{c1} and \ref{c2}.

\subsection{Examples}\label{ex} 
Recall that we are happy when Corollary \ref{c1} applies, 
since then we have an explicit central limit theorem.
Some arguments are given at the end of the paper.

\vip

\noindent (a) If $f(x)=(2\pi)^{-d/2}\exp(-|x|^2/2)$ (or any other non-degenerate normal distribution), 
Corollary \ref{c1} applies.

\vip

\noindent (b) If $f(x)=c_{d,a} e^{-(1+|x|^2)^{a/2}}$ for some $a>0$, then 
Corollary \ref{c1} applies.

\vip

\noindent (c) If $f(x)=c_{d,a}(1+|x|^2)^{-(d+a)/2}$ with $a>d$, then
Corollary \ref{c1} applies.

\vip

\noindent (d) If $f(x)=c_{d,a}|x|^{a}e^{-|x|}$ (or $f(x)=c_{a}x^{a}e^{-x}\indiq_{\{x>0\}}$ if $d=1$),
then Corollary \ref{c1} applies when $d=1, a\geq 1$, when
$d=2$, $a\geq 2$, when $d=3$, $a\in [2,12)$ and when $d\in\{4,\dots,9\}$, $a\in (d/2,4d/(d-2))$.
For any $d\geq 3$, $a\geq 2$, Corollary \ref{c2} applies.

\vip

\noindent (e) If $f(x)=c_{a,b} \prod_{i=1}^d x_i^{a_i}(1-x_i)^{b_i}\indiq_{\{x_i\in [0,1]\}}$ for some
$a=(a_1,\dots,a_d)$ and $b=(b_1,\dots,b_d)$ both in $(0,\infty)^d$,
we set $\tau=\min\{a_1,b_1,\dots,a_d,b_d\}$ and $\mu=\max\{a_1,b_1,\dots,a_d,b_d\}$.
Corollary \ref{c1} applies if $d=1$, $\tau\geq 1$ or $d=2$, $\tau\geq 2$ 
or $d=3$, $2\leq \tau \leq \mu <4$. For any $d\geq 3$, $\tau\geq 2$, Corollary \ref{c2} applies.

\vip

\noindent (f) If $d=1$ and $f(x)=c_p x^p|\sin(\pi/x)|\indiq_{\{x\in (0,1)\}}$
for some $p\geq 2$, then Corollary \ref{c1} applies.

\vip

This last example is of course far-fetched, but it shows
that our results apply to densities with many zeroes (in dimensions $1$ and $2$).

\subsection{Comparison with previous results}\label{ccc}

We will use the results of Bickel and Breiman \cite{bb}: for $\e\in(0,1)$,
for $\log_\e$ a bounded approximation of $\log$ and for $H_N^\e$ the corresponding
cutoff version of \eqref{mh}, it holds that $\sqrt N(H_N^\e-\E[H_N^\e])\to\cN(0,\sigma_\e^2(f))$.
This paper is thus very interesting and applies to very general densities $f$
but does not include the entropy as an admissible functional.
Furthermore, they do not quantify the bias.

\vip

As already mentioned, 
the results of Hall and Morton \cite{hm} apply to densities of which we know quite precisely the
tail behavior. This is a rather stringent condition.

\vip

The results of Hall \cite{h,h2}, van Es \cite{vE}, Tsybakov and van der Meulen \cite{tvdm}
and El Haje and Golubev \cite{eg} only concern the one-dimensional case. Furthermore,
\cite{h,h2,vE} apply only to densities bounded below on their (compact) support. 
A lot of regularity is assumed in \cite{tvdm}: in our list of examples, on (a)
and (b) (with $a>1$) are included. Still when $d=1$, it is difficult to compare our results
with those of \cite{eg}, because in both cases, the assumptions are not very transparent.
Let us however mention that their study seems to include example (c) for all $a>0$
(with a CLT), while we have to assume that $a>1$. On the contrary, they suppose 
that $f>0$ has a finite number of connected components, so that they cannot deal with example (f).

\vip

When $d\geq 2$, the only quantified consistency result seems to be that of P\'al,
P\'oczos and Szepesv\'ari. They assume that $f$ is compactly supported and study other notions of 
entropy, but, if we extrapolate, we find some estimate looking like $|H_N-H(f)|\leq CN^{-1/(2d)}$
with high probability. Recall that the bias is actually in $N^{-2/d}$.

\vip

As a conclusion, it seems we provide the first root $N$ and asymptotic normality
result for a general entropy estimator in dimension $d\geq 2$.

\vip

However, our assumptions are not very transparent and probably far from optimal,
at least when $d\geq 3$. Also, our proofs are rather tedious: quoting
Bickel and Breiman \cite{bb}, ``we believe this
is due to the complexity of the problem''.

\subsection{Plan of the paper}
In the next section, we compute some conditional laws and prove some
easy estimates of constant use.
Section \ref{vv} is devoted to the proof of Theorem \ref{thv} (central limit theorem).
We prove Proposition \ref{mp} in Section \ref{ee} (estimation of the variance).
In Section \ref{dd}, we study very precisely how well $f$ is approximated its mean on a small ball.
In Section \ref{bb}, we handle the proofs of Theorems \ref{thb1} and \ref{thb2} (concerning the bias).
Finally, the corollaries are verified and we discuss the examples in Section \ref{cc}.

\subsection{Notation}
We recall that we write $f$ both for the law of $X_1$ and for its density.
The functions $m$, $M$, $G_\beta$ are defined in \eqref{mM}-\eqref{gbeta}.
We introduce some shortened notation:
\begin{gather*}
\bff_1=f(X_1), \quad \bff_2=f(X_2), \quad
\bm_1=m(X_1), \quad \bm_2=m(X_2), \quad \bM_1=M(X_1),\quad \bM_2=M(X_2),\\
\ba^N_1(r)=f(B(X_1,(r/N)^{1/d})), \quad \ba^N_2(r)=f(B(X_2,(r/N)^{1/d})), \\
\bb^N_{12}(r,s)= f(B(X_1,(r/N)^{1/d})\cup B(X_2,(s/N)^{1/d})).
\end{gather*}
We write $C$ for a finite constant used in the upperbounds and $c$ for a positive constant
used in the lowerbounds. Their values do never depend on $N$, but 
are allowed to change from line to line.

\vip

{\bf During the whole proof, we assume that $r_0=1$ for simplicity.}

\section{Preliminaries}\label{pp}

To start with, we compute some conditional laws.

\begin{lem}\label{pdd}
For $N\geq 1$ and $r,s>0$, we have
\begin{gather*}
\Pr(Y^N_1>r \,|\, X_1)=(1-\ba^N_1(r))^N,\\
\Pr(Y^N_1>r,Y^N_2>s \,|\, X_1,X_2)=
\indiq_{\{|X_1-X_2|>(\frac{r\lor s}N )^{1/d}\}}
(1- \bb^N_{12}(r,s))^{N-1}.
\end{gather*}
\end{lem}

\begin{proof}
By definition, see \eqref{rz},
\begin{align*}
\{Y^N_1>r\}=\bigcap_{i=2}^{N+1} \big\{X_i \notin B(X_1,(r/N)^{1/d}) \big\}.
\end{align*} 
The first claim follows. Next,
$$
\{Y^N_1>r,Y^N_2>s\}=\big\{|X_1-X_2|>\big(\frac{r\lor s}N \big)^{1/d}\big\}\bigcap 
\bigcap_{i=3}^{N+1} \big\{X_i \notin B(X_1,(r/N)^{1/d})\cup B(X_2,(s/N)^{1/d}) \big\}.
$$
This implies the second claim.
\end{proof}

We next verify some estimates of constant use.

\begin{lem}\label{tbust}
(i) For all $N\geq 1$, all $r\in[0,N]$, we have $\ba^N_1(r)\leq v_d \bM_1 r/N$.
\vip
(ii) For all $N\geq 1$, all $r\in[0,N]$, we have $(1-\ba^N_1(r))^N\leq \exp(- \bm_1r)$.
\vip
(iii) For all $N\geq 2$, all $r\in[0,N]$, we have $(1-\ba^N_1(r))^{N-1}\leq \exp(- \bm_1r/2)$.
\vip
(iv) If $\intrd|x|^q f(x)\dd x<\infty$, put
$g(x)=1\lor \E[|X_1-x|^q]$, which is bounded by $C(1+|x|^q)$. For all $N\geq 1$,
all $r>0$, we have $1-\ba^N_1(rN)\leq \bg_1/r^{q/d}$, where $\bg_1=g(X_1)$.
\end{lem}

\begin{proof}
For (i), we use that $\sup_{B(X_1,(r/N)^{1/d})}f\leq \bM_1$ since $r\leq N$.
Consequently, $\ba^N_1(r)\leq \bM_1\Leb(B(X_1,(r/N)^d))=v_d \bM_1 r/N$.
For (ii), we write $(1-\ba^N_1(r))^N\leq \exp(-N\ba^N_1(r))$ and we use that 
$N\ba^N_1(r)\geq N \bm_1 r/N=\bm_1 r$.
Point (iii) is checked similarly, using that $(N-1)/N\geq 1/2$ for all $N\geq 2$.
By the Markov inequality, we have $1-f(B(x,r^{1/d}))
=\Pr(|X_1-x|> r^{1/d})\leq g(x)/r^{q/d}$ and (iv) follows from the fact that
$\ba^N_1(rN)=f(B(X_1,r^{1/d}))$.
\end{proof}

\section{Variance and central limit theorem}\label{vv}

This section is devoted to the proof of Theorem \ref{thv}.
We first cut $H_N$ in pieces.

\begin{lem}\label{dec}
For $\e\in(0,1]$ and $y>0$, we introduce $\log_\e y=\log(\e\lor y \land \e^{-1})$.
Then we can write, for any $N\geq 1/\e$, $H_N=H_N^\e +K_N^{1,\e}+K_N^{2,\e}+K_N^3$, where 
\begin{gather*}
H_N^\e= \frac 1 {N+1}\sum_{i=1}^{N+1} \log_\e Y^N_i +\gamma+\log v_d,\quad 
K_N^{1,\e}= \frac 1 {N+1}\sum_{i=1}^{N+1} \log [(Y^N_i/\e)\land 1], \\
K_N^{2,\e}= \frac 1 {N+1}\sum_{i=1}^{N+1} \log [1\lor (\e Y^N_i)\land (\e N)],\quad
K_N^{3}= \frac 1 {N+1}\sum_{i=1}^{N+1} \log [(Y^N_i/N)\lor 1].
\end{gather*}
\end{lem}

\begin{proof}
Recall that $N\geq 1/\e$. It suffices to note
that for all $y\in (0,\infty)$, we have 
$$\log y = \log_\e y + \log [(y/\e)\land 1]+
\log [1\lor (\e y)\land (\e N)]+\log [(y/N)\lor 1].$$
This is easily checked separating the cases $y\in (0,\e]$, $y\in (\e,1/\e]$,
$y\in(1/\e,N]$ and $y\in (N,\infty)$.
\end{proof}

We next apply the result of Bickel and Breiman \cite{bb}.

\begin{prop}\label{abb}
Assume that $f$ is bounded and continuous and fix $\e \in (0,1]$. We then have
$\sup_{N\geq 1/\e} N \Var H^\e_N <\infty$. Furthermore, 
$\sqrt N (H^\e_N-\E[H^\e_N])$ goes in law to $\cN(0,\sigma^2_\e(f))$
as $N\to \infty$, where $\sigma_\e^2(f)=A_\e+B_\e+C_\e$, with
\begin{gather*}
A_\e=\E\Big[\intoi \log_\e^2(r^d/\bff_1)\mu_0(\dd r) \Big] 
- \E\Big[\intoi \log_\e(r^d/\bff_1)\mu_0(\dd r)\Big]^2
,\\
B_\e=\E\Big[\intoi\intoi \log_\e(r_1^d/\bff_1)\log_\e(r_2^d/\bff_2)\mu_1(\dd r_1,\dd r_2)\Big],\\
C_\e=\E\Big[\intoi\intoi \log_\e(r_1^d/\bff_1)\log_\e(r_2^d/\bff_1)\mu_2(\dd r_1,\dd r_2)\Big],
\end{gather*}
where the finite (signed) measures $\mu_0$ on $[0,\infty)$ and $\mu_1,\mu_2$ on $[0,\infty)^2$
are defined by 
\begin{gather*}
\mu_0([0,r])=1-e^{-v_d r^d}, \quad 
\mu_1([0,r_1]\times[0,r_2])=e^{-v_d(r_1^d+r_2^d)}[v_d r_1^d+v_dr_2^d-v_d^2r_1^dr_2^d ],\\
\mu_2([0,r_1]\times[0,r_2])=e^{-v_d(r_1^d+r_2^d)} [T(v_d r_1^d,v_d r_2^d)-v_d(r_1\lor r_2)^d],
\end{gather*}
with $T:(0,\infty)^2\mapsto \rr_+$ defined by
$T(v_dr^d,v_ds^d)=\int_{B(0,r+s)\setminus B(0,r\lor s)} [\exp(\int_{B(0,r)\cap B(y,s)}\dd z) -1] \dd y$.
\end{prop}

\begin{proof}
Since $f:\rd\mapsto [0,\infty)$ and $\log_\e:[0,\infty)\mapsto \rr$ 
are bounded and continuous, we can apply \cite[Theorems 3.5 and 4.1]{bb} (with the notation therein,
$D_i=(Y^{N}_i)^{1/d}$, we thus take $h(x,r)=\log_\e(r^{d})$, whence $\tilde h(x,r)=\log_\e(r^d/f(x))$). 
This first theorem precisely tells us $\lim_{N\to \infty} N \Var H^\e_N =\sigma_\e^2(f)$ (whence of course
$\sup_{N\geq 1/\e} N \Var H^\e_N <\infty$) and the second one tells us that
$\sqrt N (H^\e_N-\E[H^\e_N])$ goes in law to $\cN(0,\sigma_\e^2(f))$.
Actually, there is a typo in \cite{bb}: $L_0(\dd r)$ in (3.6) has to be a non-negative measure,
so that $L_0(r)$ (see (3.1)) has to be replaced by $-L_0(r)$ or, as we did, by $1-L_0(r)$.
\end{proof}

\begin{rk}\label{rkT}
For all $r_1,r_2>0$ and $u_1,u_2>0$,
\begin{gather*}
T(v_d r_1^d,v_d r_2^d)\leq d 2^{d-1}v_d (r_1\lor r_2)^{d-1} (r_1\land r_2) e^{v_d (r_1\land r_2)^d},\\
T(u_1,u_2)\leq d 2^{d-1} (u_1\lor u_2)^{1-1/d}(u_1\land u_2)^{1/d}e^{u_1\land u_2} .
\end{gather*} 
\end{rk}

\begin{proof}
Assume $r_1\geq r_2$. We have $\int_{B(0,r_1)\cap B(y,r_2)}\dd z \leq \Leb(B(y,r_2))= v_d r_2^d$.
Since furthermore $\Leb(B(0,r_1+r_2)\setminus B(0,r_1))=v_d((r_1+r_2)^d-r_1^d)$,
$T(v_d r_1^d,v_d r_2^d)\leq v_d((r_1+r_2)^d-r_1^d)e^{v_d r_2^d}$. 
But $((r_1+r_2)^d-r_1^d)\leq d(r_1+r_2)^{d-1}r_2\leq d 2^{d-1}r_1^{d-1}r_2$, which proves the first inequality.
The second inequality follows from the first one applied to $r_1=(u_1/v_d)^{1/d}$ and  $r_2=(u_2/v_d)^{1/d}$.
\end{proof}

\begin{lem}\label{lv0}
Assume that $f$ is bounded and continuous and that $\intrd f(x)\log^2f(x)\dd x <\infty$. 
Then $\lim_{\e\to 0}\sigma^2_\e(f)=\sigma^2(f)$,
with $\sigma^2(f)$ defined in Theorem \ref{thv}.
\end{lem}

\begin{proof}
For $\xi\sim\Exp(1)$ and independent of $X_1$, it holds that $\Pr((\xi/v_d)^{1/d}>r)=\exp(-v_dr^d)$.
We thus can write $A_\e=\Var(\log_\e(\xi/(v_d\bff_1)))$. By dominated convergence, we will have
that $\lim_{\e\to 0}A_\e=\Var(\log(\xi/(v_d\bff_1)))=:A$, provided $\E[\log^2(\xi/(v_d\bff_1))]<\infty$.
This is the case, because $\intrd f(x) \log^2f(x) \dd x<\infty$ by assumption. Now by independence,
$A=\Var(\log \xi)+\Var(\log \bff_1)$. First, it holds that 
$\Var(\log \bff_1)=\intrd f(x) \log^2f(x) \dd x - (H(f))^2$.
We also have $\Var(\log \xi)=\pi^2/6$, because $\E[\log \xi]=\intoi (\log r) e^{-r}\dd r = -\gamma$
and because $\E[\log^2 \xi]=\intoi (\log^2 r) e^{-r}\dd r = \gamma^2+\pi^2/6$ (both equalities
can be found on the page {\it Euler-Mascheroni constant} of Wikipedia). 
We have proved that $\lim_{\e\to 0}A_\e=\intrd f(x) \log^2f(x) \dd x - (H(f))^2+\pi^2/6$.

\vip

An integration by parts (more precisely, writing
$\log_\e r=\log (1/\e) - \int_r^{\infty} \indiq_{\{u\in(\e,1/\e)\}} \frac{\dd u}u$, using the Fubini
theorem and that $\mu_1([0,r_1)\times[0,r_2))=0$ if $r_1\lor r_2=\infty$) shows that
$$
B_\e =\E\Big[ \intoi \intoi \frac{d \indiq_{\{r_1^d/\bff_1 \in (\e,1/\e)\}} \dd r_1}{r_1}
\frac{d \indiq_{\{r_2^d/\bff_1 \in (\e,1/\e)\}} \dd r_2}{r_2} e^{-v_d(r_1^d+r_2^d)}
[v_d r_1^d+v_dr_2^d-v_d^2r_1^dr_2^d ]\Big].
$$
Using now the change of variables $(u_1,u_2)=(v_dr_1^d,v_dr_2^d)$, we find that
\begin{align*}
B_\e= \E\Big[ \int_{v_d \bff_1 \e}^{v_d \bff_1/\e}\int_{v_d \bff_2 \e}^{v_d \bff_2/\e} e^{-u_1-u_2}(u_1+u_2-u_1u_2) 
\frac{\dd u_2}{u_2}\frac{\dd u_1}{u_1}  \Big].
\end{align*}
Similarly (observe that  $\mu_2([0,r_1)\times[0,r_2))$ is null as soon as $r_1\lor r_2=\infty$ by Remark \ref{rkT}),
\begin{align*}
C_\e=\E\Big[ \int_{v_d \bff_1 \e}^{v_d \bff_1/\e}\int_{v_d \bff_1 \e}^{v_d \bff_1/\e} e^{-u_1-u_2}(T(u_1,u_2)
-u_1\lor u_2) 
\frac{\dd u_2}{u_2}\frac{\dd u_1}{u_1}  \Big].
\end{align*}
Unfortunately, neither $B_\e$ nor $C_\e$ converge as $\e\to 0$.
We write $B_\e+C_\e=I_\e+J_\e+K_\e$, where
\begin{gather*}
I_\e=\E\Big[ \int_{v_d \bff_1 \e}^{v_d \bff_1/\e}\int_{v_d \bff_2 \e}^{v_d \bff_2/\e} e^{-u_1-u_2}(u_1+u_2-u_1\lor u_2-u_1u_2)
\frac{\dd u_2}{u_2}\frac{\dd u_1}{u_1}  \Big],\\
J_\e=\E\Big[ \int_{v_d \bff_1 \e}^{v_d \bff_1/\e}\int_{v_d \bff_1 \e}^{v_d \bff_1/\e} e^{-u_1-u_2}T(u_1,u_2)
\frac{\dd u_2}{u_2}\frac{\dd u_1}{u_1}  \Big],\\
K_\e=\E\Big[\int_{v_d \bff_1 \e}^{v_d \bff_1/\e}\Big(\int_{v_d \bff_2 \e}^{v_d \bff_2/\e} e^{-u_2}(u_1\lor u_2) 
\frac{\dd u_2}{u_2}-
\int_{v_d \bff_1 \e}^{v_d \bff_1/\e} e^{-u_2}(u_1\lor u_2) 
\frac{\dd u_2}{u_2}\Big)e^{-u_1}\frac{\dd u_1}{u_1}  \Big].
\end{gather*}

Since $|u_1+u_2-u_1\lor u_2-u_1u_2|=|u_1\land u_2 - u_1u_2|\leq u_1u_2+\sqrt{u_1u_2} 
\in L^1(\rr_+^2,e^{-u_1-u_2}\frac{\dd u_1}{u_1}\frac{\dd u_2}{u_2})$,
we see that $\lim_{\e\to 0} I_\e=\intoi\intoi e^{-u_1-u_2}(u_1+u_2-u_1\lor u_2-u_1u_2)\frac{\dd u_2}{u_2}
\frac{\dd u_1}{u_1}=:I$ by dominated convergence.
But $I=2\intoi\int_0^{u_1} e^{-u_1-u_2}u_2(1-u_1)\frac{\dd u_2}{u_2}\frac{\dd u_1}{u_1}=
2 \intoi e^{-u_1}(1-e^{-u_1})(1-u_1)\frac{\dd u_1}{u_1}=2\log 2 -1$.
Indeed, introduce $\varphi(t)=2 \intoi e^{-u_1}(1-e^{-t u_1})(1-u_1)\frac{\dd u_1}{u_1}$:
we have $\varphi(0)=0$ and $\varphi'(t)=2 \intoi e^{-u_1}e^{-t u_1}(1-u_1)\dd u_1
=2 t/(1+t)^2$, so that $\varphi(1)=\int_0^1 [2 t/(1+t)^2]\dd t=2\log 2 -1$.

\vip

Next, we deduce from Remark \ref{rkT} that 
$T\in  L^1(\rr_+^2,e^{-u_1-u_2}\frac{\dd u_1}{u_1}\frac{\dd u_2}{u_2})$,
whence $\lim_{\e\to 0}J_\e= \intoi\intoi e^{-u_1-u_2}T(u_1,u_2)\frac{\dd u_2}{u_2}\frac{\dd u_1}{u_1}$.

\vip

Finally, we verify that $\lim_{\e\to 0} K_\e=0$. Using a symmetry argument (and that $\bff_1,\bff_2$ are i.i.d.),
$$
K_\e=2\E\Big[\int_{v_d \bff_1 \e}^{v_d \bff_1/\e} \Big(\int_{v_d \bff_2 \e}^{v_d \bff_2/\e} e^{-u_2} \indiq_{\{u_2<u_1\}}
\frac{\dd u_2}{u_2}- \int_{v_d \bff_1 \e}^{v_d \bff_1/\e} e^{-u_2} \indiq_{\{u_2<u_1\}}
\frac{\dd u_2}{u_2}\Big) e^{-u_1} \dd u_1  \Big].
$$
Since next $\bff_1$ and $\bff_2$ have the same law, it holds that
$$
L_\e=2\E\Big[\intoi \Big(\int_{v_d \bff_2 \e}^{v_d \bff_2/\e} e^{-u_2} \indiq_{\{u_2<u_1\}}
\frac{\dd u_2}{u_2}- \int_{v_d \bff_1 \e}^{v_d \bff_1/\e} e^{-u_2} \indiq_{\{u_2<u_1\}}
\frac{\dd u_2}{u_2}\Big) e^{-u_1} \dd u_1  \Big]=0.
$$
Hence $|K_\e|=|K_\e-L_\e|\leq 2 \E[\int_0^{v_d \bff_1 \e}|F_\e(u_1,\bff_1,\bff_2) | e^{-u_1} \dd u_1 
+\int_{v_d \bff_1 /\e}^\infty |F_\e(u_1,\bff_1,\bff_2)|  e^{-u_1} \dd u_1]$, where
$F_\e(u_1,\bff_1,\bff_2)=\int_{v_d \bff_2 \e}^{v_d \bff_2/\e} e^{-u_2} \indiq_{\{u_2<u_1\}}
\frac{\dd u_2}{u_2}- \int_{v_d \bff_1 \e}^{v_d \bff_1/\e} e^{-u_2} \indiq_{\{u_2<u_1\}}
\frac{\dd u_2}{u_2}$. But 
$$
|F_\e(u_1,\bff_1,\bff_2)|\leq {2+|\log(v_d \bff_1 \e)|+|\log(v_d \bff_2\e)|}
\leq C(1+|\log \bff_1|+|\log \bff_2|+|\log \e|)
$$
and we end with
$$
K_\e\leq C \E\Big[ (1-e^{-v_d\bff_1 \e}+e^{-v_d\bff_1/\e})(1+|\log \bff_1|+|\log \bff_2|+|\log \e|) \Big].
$$
Using that $f$ is bounded, we see that $1-e^{-v_d\bff_1 \e}\leq C \e$ so that,
since $\E[|\log \bff_1|]<\infty$ by assumption, we clearly have 
$\lim_{\e\to 0}\E[ (1-e^{-v_d\bff_1 \e})(1+|\log \bff_1|+|\log \bff_2|+|\log \e|)]=0$.
Next,
$\lim_{\e\to 0}\E[ e^{-v_d\bff_1 /\e}(1+|\log \bff_1|+|\log \bff_2|)]=0$ by dominated convergence.
At last,
$$
\E[ e^{-v_d\bff_1 /\e}]\leq e^{-v_d/\sqrt\e}+\Pr\Big(\bff_1 < \sqrt \e\Big)
\leq e^{-v_d/\sqrt\e}+\Pr\Big(|\log \bff_1|> |\log \sqrt\e|\Big)
\leq e^{-v_d/\sqrt\e}+\frac{4\E[\log^2\bff_1]}{\log^2 \e}.
$$
We conclude that $\lim_{\e\to 0} |\log \e|\E[ e^{-v_d\bff_1 /\e}]=0$, so that $\lim_{\e\to 0}K_\e=0$.

\vip

Recalling that $\sigma^2_\e(f)=A_\e+B_\e+C_\e=A_\e+I_\e+J_\e+K_\e$, we have checked that 
$\lim_{\e\to 0} \sigma_\e^2(f)=\intrd f(x) \log^2f(x) \dd x - (H(f))^2+\pi^2/6
+2\log 2 -1 + \intoi \intoi e^{-u_1-u_2}T(u_1,u_2)
\frac{\dd u_2}{u_2}\frac{\dd u_1}{u_1}$, which is nothing but $\sigma^2(f)$.
\end{proof}

It remains to show that $K^{1,\e}_N$, $K^{2,\e}_N$ and $K^{3}_N$ are sufficiently small.
We first give some expressions of their variances.

\begin{lem}\label{expr2}
Recall that $K^{1,\e}_N$, $K^{2,\e}_N$ and $K^{3}_N$ were defined, for $\e\in(0,1]$ and
$N\geq 1/\e$, in Lemma \ref{dec}. For $r,s \geq 0$, put
$$
\Gamma_N(r,s) =\indiq_{\{|X_1-X_2|>(\frac{r\lor s}N )^{1/d}\}}
(1- \bb^N_{12}(r,s))^{N-1}-(1-\ba^N_1(r))^N(1-\ba^N_2(s))^N.
$$
We have $\Var K_N^{1,\e}\leq W^{11,\e}_N+W^{12,\e}_N$, $\Var K_N^{2,\e}\leq W^{21,\e}_N+W^{22,\e}_N$
and $\Var K_N^{3}\leq W^{31}_N+W^{32}_N$, where 
\begin{gather*}
W^{11,\e}_N=\frac 2{N+1} \int_0^1 \E[1-(1-\ba^N_1(\e r))^N] \log \frac 1r \frac{\dd r}r,\quad
W^{12,\e}_N=\frac{2N}{N+1}\int_{0}^1\int_0^r \E[\Gamma_N(\e r,\e s)] \frac{\dd s}s \frac{\dd r}r,\\
W^{21,\e}_N= \frac 2{N+1} \int_1^{\e N} \E[(1-\ba^N_1(r/\e ))^N] \log r \frac{\dd r}r,\quad
W^{22,\e}_N=\frac{2N}{N+1}\int_{1}^{\e N}\int_{1}^r \E[\Gamma_N(r/\e,s/\e)] \frac{\dd s}s \frac{\dd r}r,\\
W^{31}_N= \frac{2}{N+1}\int_1^\infty \E[(1-\ba^N_1(Nr))^N] \log r \frac{\dd r}r,\quad
W^{32}_N=\frac{2N}{N+1} \int_1^\infty\int_1^r \E[\Gamma_N(Nr,Ns)] \frac{\dd s}s \frac{\dd r}r.
\end{gather*}
\end{lem}

\begin{proof}
By exhangeability, we have $\Var K_N^{1,\e}\leq W^{11,\e}_N+W^{12,\e}_N$, $\Var K_N^{2,\e}\leq W^{21,\e}_N+W^{22,\e}_N$
and $\Var K_N^{3}\leq W^{31}_N+W^{32}_N$, where 
\begin{gather*}
W^{11,\e}_N=(N+1)^{-1} \E[\log^2 [(Y^N_1/\e)\land 1]], \\
W^{21,\e}_N=(N+1)^{-1} \E[\log^2 [1\lor (\e Y^N_1)\land (\e N)]]\\
W^{31}_N=(N+1)^{-1} \E[\log^2 [(Y^N_1/N)\lor 1]],\\
W^{12,\e}_N=
(N+1)^{-1}N \Cov (\log [(Y^N_1/\e)\land 1],\log [(Y^N_2/\e)\land 1]),\\
W^{22,\e}_N=
(N+1)^{-1}N \Cov (\log [1\lor (\e Y^N_1)\land (\e N)],\log [1\lor (\e Y^N_2)\land (\e N)]),\\
W^{32}_N=
(N+1)^{-1}N \Cov (\log [(Y^N_1/N)\lor 1],\log [(Y^N_2/N)\lor 1]).
\end{gather*}

Since
$\log^2 z =
2\int_1^\infty \indiq_{\{r<z\}} \log r\frac{\dd r}r +2\int_0^1 \indiq_{\{r\geq z\}} \log \frac 1r\frac{\dd r}r$ 
for all $z>0$, it holds that for 
$Z$ a positive random variable,
\begin{align}\label{ss1}
\E[\log^2 Z]=2\int_1^\infty \Pr(Z> r)\log r \frac{\dd r}r
+2\int_0^1\Pr(Z\leq r)\log \frac1r \frac{\dd r}r.
\end{align}
Using Lemma \ref{pdd}, we easily conclude that
\begin{gather*}
W^{11,\e}_N = \frac 2{N+1} \int_0^1 \Pr(Y^N_1\leq \e r)\log \frac 1 r \frac{\dd r}r
=\frac 2{N+1} \int_0^1 \E[1-(1-\ba^N_1(\e r))^N]\log \frac 1 r \frac{\dd r}r,\\
W^{21,\e}_N = \frac 2{N+1} \int_1^\infty \Pr((\e Y^N_1)\land (\e N) > r)\log r \frac{\dd r}r
=\frac 2{N+1} \int_1^{\e N} \E[(1-\ba^N_1(r/\e))^N]\log r \frac{\dd r}r,\\
W^{31}_N= \frac 2{N+1} \int_1^\infty \Pr(Y^N_1/N > r)\log r \frac{\dd r}r
=\frac 2{N+1} \int_1^{\infty} \E[(1-\ba^N_1(Nr))^N]\log r \frac{\dd r}r.
\end{gather*}

Also, for $Z_1$ and $Z_2$ two positive random variables,
\begin{align}\label{ss2}
\Cov (\log Z_1, \log Z_2 )=\intoi\intoi \Big(\Pr(Z_1>r,Z_2>s)-\Pr(Z_1>r)\Pr(Z_2>s) \Big) \frac{\dd s}s
\frac{\dd r}r.
\end{align}
Indeed, it suffices to use that $\log z =\intoi (\indiq_{\{r<z\}}-\indiq_{\{r<1\}})\frac{\dd r}r$ and
the bilinearity of the covariance. For $r,s>0$, we deduce from Lemma \ref{pdd} 
(and the independance of $\ba^N_1(r)$ and  $\ba^N_2(s)$) that
\begin{align}\label{ttiicc}
&\Pr(Y_1^N>r,Y_2^N>s)-\Pr(Y_1^N>r)\Pr(Y_2^N>s)\\
=& \E\Big[\indiq_{\{|X_1-X_2|>(\frac{r\lor s}N )^{1/d}\}}
(1- \bb^N_{12}(r,s))^{N-1}\Big]-\E[(1-\ba^N_1(r))^N]\E[(1-\ba^N_2(s))^N] \nonumber\\
=& \E[\Gamma_N(r,s)].\nonumber
\end{align}
Consequently, $\Pr((Y^N_1/\e)\land 1>r,(Y^N_2/\e)\land 1>s)-\Pr((Y^N_1/\e)\land 1>r)\Pr((Y^N_2/\e)\land 1>s)
=\E[\Gamma_N(\e r,\e s)]\indiq_{\{r,s\in [0,1]\}}$, whence by \eqref{ss2},
$$
W^{12,\e}_N =\frac N{N+1} \int_0^1\int_0^1\E[\Gamma_N(\e r,\e s)]\frac{\dd s}s\frac{\dd r}r
=\frac {2N}{N+1} \int_0^1\int_0^r\E[\Gamma_N(\e r,\e s)]\frac{\dd s}s\frac{\dd r}r.
$$
Similarly, we deduce from \eqref{ttiicc} the equality
$\Pr(1\lor (\e Y^N_1)\land (\e N)>r,1\lor (\e Y^N_2)\land (\e N)>s)-{\Pr(1\lor (\e Y^N_1)\land (\e N)>r)
\Pr(1\lor (\e Y^N_2)\land (\e N)>s)}
=\E[\Gamma_N(r/\e,s/\e)]\indiq_{\{r,s\in [1,\e N]\}}$, so that
$$
W^{22,\e}_N =\frac N{N+1} \int_1^{\e N}\int_1^{\e N}\E[\Gamma_N(r/\e,s/\e)]\frac{\dd s}s\frac{\dd r}r
=\frac {2N}{N+1} \int_1^{\e N}\int_1^{r}\E[\Gamma_N(r/\e,s/\e)]\frac{\dd s}s\frac{\dd r}r.
$$
Finally, $\Pr((Y^N_1/N)\lor 1>r,(Y^N_2/N)\lor 1>s)-{\Pr((Y^N_1/N)\lor 1>r)\Pr((Y^N_2/N)\lor 1>s)}
=\E[\Gamma_N(Nr,Ns)]\indiq_{\{r,s\geq 1\}}$, whence
$$
W^{32}_N =\frac N{N+1} \int_1^{\infty}\int_1^{\infty}\E[\Gamma_N(N r,N s)]\frac{\dd s}s\frac{\dd r}r
=\frac {2N} {N+1} \int_1^{\infty}\int_1^r\E[\Gamma_N(N r,N s)]\frac{\dd s}s\frac{\dd r}r.
$$
The proof is complete.
\end{proof}

We now study the terms $W_N^{11,\e}$, $W_N^{12,\e}$, $W_N^{21,\e}$, $W_N^{22,\e}$, $W_N^{31}$ and $W_N^{32}$.

\begin{lem}\label{lv1}
Assume that $f$ is bounded. Then $N W^{11,\e}_N\leq C \e $ for all $\e\in(0,1]$ and all $N\geq 1/\e$.
\end{lem}

\begin{proof}
Recall that $W^{11,\e}_N= 2(N+1)^{-1}\E[\int_0^1 (1-(1-\ba^N_1(\e r))^N) \log \frac 1r \frac{\dd r}r]$.
But $1-(1-x)^N\leq Nx$ for $x\in[0,1]$, so that
$1-(1-\ba^N_1(\e r))^N \leq N \ba^N_1(\e r) \leq v_d \bM_1 \e r\leq C \e r$ by  Lemma \ref{tbust}-(i)
and since $f$ is bounded. Thus
\begin{align*}
W^{11,\e}_N \leq \frac {C\e} {N+1}\int_0^1 \log \frac 1r \dd r \leq \frac {C\e}N
\end{align*}
as desired.
\end{proof}

\begin{lem}\label{lv2}
Assume that $\intrd [\log^2 m(x)] f(x)\dd x<\infty$. Then we have $N W^{21,1}_N\leq C$ for all $N\geq 1$
and $\lim_{\e\to 0}\sup_{N\geq 1/\e} NW^{21,\e}_N=0$.
\end{lem}

\begin{proof}
Recall that $W^{21,\e}_N= 2(N+1)^{-1}\E[\int_1^{\e N} (1-\ba^N_1(r/\e))^N \log r \frac{\dd r}r]$.
Thanks to Lemma \ref{tbust}-(ii) (and since $r/\e \leq N$), 
\begin{align*}
NW^{21,\e}_N\leq 2 \E\Big[\int_1^\infty e^{- \bm_1 r/\e}  \log r \frac{\dd r}r\Big].
\end{align*}
If $\e=1$, we use that $2\int_1^\infty e^{-ar}\log r \frac{\dd r}r 
=\int_1^\infty ae^{-ar}\log^2 r \dd r=\int_a^\infty e^{-t} \log^2(t/a)\dd t
\leq C(1+\log^2 a)$ and that $\E[\log^2\bm_1]<\infty$ by assumption to conclude that
$\sup_{N\geq 1} NW_N^{21,1}<\infty$. The fact that $\lim_{\e\to 0}\sup_{N\geq 1/\e} NW^{21,\e}_N=0$ follows
from the dominated convergence theorem.
\end{proof}

\begin{lem}\label{lv3}
If $\intrd (|x|^q+[\log^2 (2+|x|)][f(B(x,1))]^{-\theta})f(x)\dd x<\infty$ 
for some $q>0$ and some $\theta>0$, then  $NW^{31}_N\leq C N^{-\theta}$ for all $N\geq 1$.
\end{lem}

\begin{proof}
Recall that $W^{31}_N=2(N+1)^{-1}\E[\int_1^\infty (1-\ba^N_1(Nr))^N \log r \frac{\dd r}r]$.
For $g(x)=1\lor \E[|X_1-x|^q]$ as in Lemma \ref{tbust}, we write $NW^{31}_N\leq 2A_N+2B_N$, where
$$
A_N=\E\Big[\int_1^{(2\bg_1)^{d/q}}  (1-\ba^N_1(Nr))^N \log r \frac{\dd r}r\Big]
\quad\hbox{and}\quad
B_N=\E\Big[\int_{(2\bg_1)^{d/q}}^\infty  (1-\ba^N_1(Nr))^N \log r \frac{\dd r}r\Big].
$$

Using that $(1-\ba^N_1(rN))^N\leq (\bg_1/r^{q/d})^N$ by Lemma \ref{tbust}-(iv),
$$
B_N\leq \E\Big[\int_{(2\bg_1)^{d/q}}^\infty  \Big(\frac{\bg_1}{r^{q/d}}\Big)^N \log r \frac{\dd r}r \Big]
\leq \Big(\frac 12\Big)^{N-1} \E\Big[ \bg_1 \int_1^\infty \frac{\log r}{r^{1+q/d}}\dd r\Big]\leq \frac C{2^N}.
$$
because $\E[\bg_1]<\infty$ (recall that $g(x)\leq C(1+|x|^q)$).

\vip

Next, there is a constant $C$ such that $(1-x)^N\leq e^{-Nx}\leq C (Nx)^{-\theta}$ for all $x\in (0,1]$,
whence $(1-\ba^N_1(rN))^N\leq C (N\ba^N_1(rN))^{-\theta}\leq C(N\ba^N_1(N))^{-\theta}=C(Nf(B(X_1,1)))^{-\theta}$
for all $r\geq 1$. Thus
\begin{align*}
A_N \leq& \frac C{N^{\theta}}\E\Big[ \int_1^{(2\bg_1)^{d/q}} \frac 1 {(f(B(X_1,1)))^\theta} \log r \frac{\dd r}r\Big]
\leq \frac C{N^{\theta}}\E\Big[\frac{\log^2 (2\bg_1)^{d/q}}{(f(B(X_1,1)))^\theta}\Big] \leq \frac C{N^{\theta}}
\end{align*}
because $\E[(f(B(X_1,1)))^{-\theta}\log^2(2\bg_1)^{d/q}]<\infty$ 
(observe that $\log^2 (2g(x))^{d/q} \leq C \log^2(2+|x|)$).
All in all, $NW^{31}_N\leq C2^{-N}+CN^{-\theta}\leq CN^{-\theta}$ as desired.
\end{proof}

\begin{lem}\label{lv4}
If $\intrd (|x|^q+[\log^2 (2+|x|)][f(B(x,1))]^{-\theta})f(x)\dd x<\infty$ 
for some $q>0$ and some $\theta>0$, 
then  $NW^{32}_N\leq C N^{1-2\theta}$ for all $N\geq 2$ (we do not claim
that $|N W^{32}_N|\leq C N^{1-2\theta}$).
\end{lem}

\begin{proof}
Recall that $W^{32}_N=2(N/(N+1))\int_1^\infty\int_1^r \E[\Gamma_N(Nr,Ns)] \frac{\dd s}{s}\frac{\dd r} r$
and that we have the obvious bound $\Gamma_N(Nr,Ns)\leq (1-\bb^N_{12}(Nr,Ns))^{N-1}$.
Introducing  $g(x)=1\lor \E[|X_1-x|^q]$ as in Lemma \ref{tbust}, we write
$W^{32}_N\leq 2(A_N+B_N)$, where
\begin{gather*}
A_N= \E\Big[\int_1^{(2\bg_1)^{d/q}}\int_1^s(1-\bb^N_{12}(rN,sN))^{N-1}
\frac{\dd s}s\frac{\dd r}r\Big],\\
B_N= \E\Big[\int_{(2\bg_1)^{d/q}}^\infty\int_1^s(1-\bb^N_{12}(rN,sN))^{N-1}
\frac{\dd s}s\frac{\dd r}r\Big].
\end{gather*}

First, we use that $1-\bb^N_{12}(rN,sN)\leq 1-\ba^N_{1}(rN)\leq \bg_1r^{-q/d}$  by Lemma \ref{tbust}-(iv), whence
$$
B_N \leq  \E\Big[\int_{(2\bg_1)^{d/q}}^\infty \Big(\frac {\bg_1}{r^{q/d}}\Big)^{N-1}
\log r \frac{\dd r}r\Big]\leq \frac C {2^{N}}
$$
as in the previous proof (here $N\geq 2$).

\vip

Next, there is a constant $C$ such that $(1-x)^{N-1}\leq e^{-(N-1)x}
\leq C/(Nx)^{2\theta}$ for all $N\geq 2$, all $x\in (0,1]$.
Moreover, for all $r,s \geq 1$,
$$
\bb^N_{12}(Nr,Ns) \geq f(B(X_1,1))\lor f(B(X_2,1))
\geq [f(B(X_1,1)) f(B(X_2,1))]^{1/2}.
$$
Thus
$(1-\bb^N_{12}(Nr,Ns))^{N-1} \leq CN^{-2\theta} 
[f(B(X_1,1))f(B(X_2,1))]^{-\theta}$. Since  $2\int_1^a \int_1^r \frac{\dd s}{s}\frac{\dd r} r=
\log^2 a$,
$$
A_N \leq \frac C{N^{2\theta}} \E\Big[\frac{\log^2 (2\bg_1)^{d/q}}{[f(B(X_1,1))f(B(X_2,1))]^{\theta}}\Big]
\leq \frac C{N^{2\theta}} \E\Big[\frac{\log^2 (2+|X_1|)}{[f(B(X_1,1))]^{\theta}}\Big] 
\E\Big[\frac 1{[f(B(X_2,1))]^{\theta}}\Big] \leq \frac C{N^{2\theta}}.
$$
We used that $\log^2(2g(x))^{d/q}\leq C\log^2(2+|x|)$, recall that $g(x)\leq C(1+|x|^q)$. 
All in all, $W^{32}_N\leq C N^{-2\theta} + C 2^{-N}\leq C N^{-2\theta} $.
\end{proof}

Finally, we treat $W^{31,\e}_N$ and $W^{32,\e}_N$, which are more difficult.

\begin{lem}\label{lv5}
If $\intrd (M(x)/m(x))(1+|\log m(x)|)f(x)\dd x<\infty$,
then $N(W^{12,1}_N+W^{22,1})\leq C$ for all $N\geq 2$ and $\limsup_{\e\to 0}\sup_{N\geq 1/\e} N(W^{12,\e}_N
+W^{22,\e}_N)\leq 0$.
\end{lem}

\begin{proof}
Recall that $\Gamma_N(r,s) =\indiq_{\{|X_1-X_2|>(\frac{r\lor s}N )^{1/d}\}}
(1- \bb^N_{12}(r,s))^{N-1}-(1-\ba^N_1(r))^N(1-\ba^N_2(s))^N$ and that,
for $N\geq 1/\e$,
\begin{gather*}
W^{12,\e}_N=\frac{2N}{N+1}\int_0^1\int_0^r \E[\Gamma_N(\e r,\e s)] \frac{\dd s}{s}\frac{\dd r} r
=\frac{2N}{N+1}\int_0^\e\int_0^r \E[\Gamma_N(r,s)] \frac{\dd s}{s}\frac{\dd r} r,\\
W^{22,\e}_N=\frac{2N}{N+1}\int_1^{\e N}\int_1^r \E[\Gamma_N(r/\e,s/\e)] \frac{\dd s}{s}\frac{\dd r} r
= \frac{2N}{N+1}\int_{1/\e}^{N}\int_{1/\e}^r \E[\Gamma_N(r,s)] \frac{\dd s}{s}\frac{\dd r} r.
\end{gather*}
We observe that $|X_1-X_2|>(\frac r N )^{1/d}+ (\frac s N)^{1/d}$ implies
$B(X_1,(r/N)^{1/d})\cap B(X_2,(s/N)^{1/d})=\emptyset$ and thus $\bb^N_{12}(r,s)=\ba^N_1(r)+\ba^N_2(s)$.
Hence if $r>s>0$, we have $\Gamma_N(r,s)=\Gamma_N^1(r,s)+\Gamma_N^2(r,s)$,
where
\begin{align*}
\Gamma_N^1(r,s) =&\indiq_{\{|X_1-X_2|>(\frac r N )^{1/d}+ (\frac s N)^{1/d}\}}
(1- \ba^N_{1}(r)-\ba^N_2(s))^{N-1}-(1-\ba^N_1(r))^N(1-\ba^N_2(s))^N,\\
\Gamma_N^2(r,s) =&\indiq_{\{(\frac{r}N )^{1/d}<|X_1-X_2|\leq(\frac r N )^{1/d}+ (\frac s N)^{1/d}\}}
(1- \bb^N_{12}(r,s))^{N-1}-(1-\ba^N_1(r))^N(1-\ba^N_2(s))^N.
\end{align*}
Using next that 
$(1-\ba^N_1(r)-\ba^N_2(s))\leq (1-\ba^N_1(r))(1-\ba^N_2(s))$, we may write, if $r> s >0$,
\begin{align*}
\Gamma_N^1(r,s)\leq& \indiq_{\{|X_1-X_2|>(\frac r N )^{1/d}\}}
(1- \ba^N_{1}(r))^{N-1}(1-\ba^N_2(s))^{N-1}-(1-\ba^N_1(r))^N(1-\ba^N_2(s))^N\\
\leq& (1- \ba^N_{1}(r))^{N-1}(1-\ba^N_2(s))^{N-1}
\Big[\ba^N_1(r)+\ba^N_2(s)- \indiq_{\{|X_1-X_2|\leq(\frac r N )^{1/d}\}}\Big]\\
=&\Gamma_N^{11}(r,s)+\Gamma_N^{12}(r,s)+\Gamma_N^{13}(r,s)+\Gamma_N^{14}(r,s),
\end{align*}
where
\begin{align*}
\Gamma_N^{11}(r,s)=&(1- \ba^N_{1}(r))^{N-1}(1-\ba^N_2(s))^{N-1} \ba^N_2(s) ,\\
\Gamma_N^{12}(r,s)=&(1- \ba^N_{1}(r))^{N-1}[(1- \ba^N_2(s))^{N-1}- (1-\ba^N_1(s))^{N-1}]\ba^N_{1}(r),\\
\Gamma_N^{13}(r,s)=&(1- \ba^N_{1}(r))^{N-1}(1- \ba^N_1(s))^{N-1}[\ba^N_{1}(r)-\indiq_{\{|X_1-X_2|\leq(\frac r N )^{1/d}\}}] ,\\
\Gamma_N^{14}(r,s)=&(1- \ba^N_{1}(r))^{N-1}[(1- \ba^N_1(s))^{N-1}- (1- \ba^N_2(s))^{N-1}]
\indiq_{\{|X_1-X_2|\leq (\frac r N )^{1/d}\}}.
\end{align*}

{\it Step 1.} 
We first study $\Gamma^{2}_N$. For $0<s<r<N$, since $\bb^N_{12}(r,s)\geq \ba^N_1(r)$, 
\begin{align*}
\Gamma_N^2(r,s) \leq \indiq_{\{(\frac{r}N )^{1/d}<|X_1-X_2|\leq(\frac r N )^{1/d}+ (\frac s N)^{1/d}\}}
(1- \ba^N_{1}(r))^{N-1} \leq \indiq_{\{X_2 \in C_N(X_1,r,s)\}}e^{- \bm_1 r/2}.
\end{align*}
We introduced the annulus $C_N(x,r,s)=B(x,(r/N)^{1/d}+ (s/N)^{1/d})
\setminus B(x,(r/N)^{1/d})$ and we used Lemma \ref{tbust}-(iii).
Consequently,
\begin{align*}
\E[\Gamma_N^2(r,s)] \leq \E\Big[f(C_N(X_1,r,s)) e^{- \bm_1 r/2}\Big]
\leq \frac{v_d}N\E\Big[\bM_1 e^{-\bm_1 r/2}\Big]\Big((r^{1/d}+s^{1/d})^d - r \Big),
\end{align*}
because $\sup_{C_N(X_1,r,s)}f \leq \bM_1$ (since $(r/N)^{1/d}+ (s/N)^{1/d}\leq 2$)
and because $\Leb(C_N(X_1,r,s))=(v_d/N)((r^{1/d}+s^{1/d})^d - r )$. But (recall that $r>s>0$)
\begin{align*}
(r^{1/d}+s^{1/d})^d - r\leq d s^{1/d}(r^{1/d}+s^{1/d})^{d-1}\leq d s^{1/d}(2r^{1/d})^{d-1}. 
\end{align*}
We have checked that $\E[\Gamma_N^2(r,s)] \leq CN^{-1}\E[\bM_1e^{- \bm_1 r/2}]s^{1/d}r^{(d-1)/d}$
for all $0<s<r<N$.

\vip

{\it Step 2.}
Next, $\Pr( |X_1-X_2|\leq (r /N )^{1/d}\,|\, X_1)=\ba^N_{1}(r)$, so that
$\E[\Gamma_N^{13}(r,s)]=0$.

\vip

{\it Step 3.}
Lemma \ref{tbust}-(i)-(iii) gives us 
$\E[\Gamma_N^{11}(r,s)]\leq v_d N^{-1}\E[\bM_2 \exp(- \bm_1 r/2-\bm_2 s/2)] s$.

\vip

{\it Step 4.} Since $(1-x)^{N-1}-(1-y)^{N-1}\leq\indiq_{\{x\leq y\}}
(N-1)(1-x)^{N-2}(y-x)\leq N y$ for all $x,y\in[0,1]$,
\begin{align*}
\Gamma_N^{12}(r,s)\leq N (1-\ba^N_1(r))^{N-1}\ba^N_1(s)\ba^N_1(r)\leq v_d (1-\ba^N_1(r))^{N-1}\ba^N_1(r) \bM_1 s .
\end{align*}
by Lemma \ref{tbust}-(i). But since $N\geq 2$,
$$
(1-\ba^N_1(r))^{N-1}\ba^N_1(r)\leq \ba^N_1(r)e^{-N\ba^N_1(r)/2}
\leq N^{-1}[\bm_1 r e^{-\bm_1 r/2} + \indiq_{\{\bm_1 r<2\}}].
$$
We finally used that $N\ba^N_1(r)\geq  \bm_1 r$
and that the function $f(x)=xe^{-x/2}$ is bounded by $1$ and decreasing on $[2,\infty)$.
Thus $\E[\Gamma_N^{12}(r,s)]\leq v_d  N^{-1}\E[(\bm_1 r e^{-\bm_1 r/2} + \indiq_{\{\bm_1 r<2\}})\bM_1 ]s$.

\vip

{\it Step 5.} Since $(1-x)^{N-1}-(1-y)^{N-1}\leq N y$ for all $x,y\in[0,1]$ as in Step 4,
$$
\Gamma_N^{14}(r,s) \leq N(1-\ba^N_1(r))^{N-1}\ba^N_2(s)\indiq_{\{|X_1-X_2|\leq(\frac rN)^{1/d}\}}.
$$
But $|X_1-X_2|\leq(r/N)^{1/d}\leq 1$ implies that $\sup_{B(X_2,(s/N)^{1/d})}f \leq \sup_{B(X_1,2)}f \leq
\bM_1$, whence $\ba^N_2(s)\leq \bM_1\Leb(B(X_2,(s/N)^{1/d})) = v_d N^{-1} \bM_1 s$.
Consequently,
\begin{align*}
\Gamma_N^{14}(r,s)\leq & v_d  (1-\ba^N_1(r))^{N-1} \bM_1 s \indiq_{\{|X_1-X_2|\leq(r/N)^{1/d}\}}.
\end{align*}
Taking first the expectation knowing $X_1$, we find
$\E[\Gamma_N^{14}(r,s)] \leq  v_d\E[(1-\ba^N_1(r))^{N-1} \ba^N_1(r) \bM_1 ]s$.
We thus conclude exactly as in Step 4 that
$\E[\Gamma_N^{14}(r,s)]\leq v_d  N^{-1}\E[(\bm_1 r e^{-\bm_1 r/2} + \indiq_{\{\bm_1 r<2\}})\bM_1]s$.

\vip

{\it Step 6.} Gathering the bounds found in the five first steps, we see that
$\E[\Gamma_N(r,s)] \leq C N^{-1} F(r,s)$ for all $0<s<r<N$, where
$$
F(r,s)=\E\Big[\bM_1e^{- \bm_1 r/2}s^{1/d}r^{(d-1)/d}+\bM_2 e^{- \bm_1 r/2-\bm_2 s/2}s+
(\bm_1 r e^{-\bm_1 r/2} + \indiq_{\{\bm_1 r<2\}})\bM_1 s\Big].
$$ 
Thus for all $\e\in(0,1]$, all $N\geq (1/\e)\lor 2$,
$$
N(W^{12,\e}_N+W^{22,\e}_N)\leq C \int_0^\infty \int_0^r \frac{F(r,s)}{rs}
\indiq_{\{r\leq \e \hbox{ or }r\geq 1/\e\}} \dd s \dd r.
$$
If we show that 
$\int_0^\infty \int_0^r F(r,s) \frac{\dd s}{s}\frac{\dd r}{r}<\infty$, this will imply
that $\sup_{N\geq 2}N(W^{12,1}_N+W^{22,1}_N)<\infty$ and also that 
$\limsup_{\e \to 0}\sup_{N\geq 1/\e}N(W^{12,\e}_N+W^{22,\e}_N)\leq 0$ by dominated convergence.
First,
\begin{align*}
\int_0^r \frac{F(r,s)}{rs}\dd s \leq & C \E\Big[\bM_1 e^{-\bm_1 r/2} + \frac{\bM_2}{\bm_2 r}e^{-\bm_1 r /2}(1-e^{-\bm_2 r /2})
+\bM_1 \bm_1 r e^{-\bm_1 r/2} + \bM_1 \indiq_{\{\bm_1 r < 2\}}    \Big]\\
\leq & C \E\Big[\bM_1 e^{-\bm_1 r/4} + \frac{\bM_2}{\bm_2 r}e^{-\bm_1 r /2}(1-e^{-\bm_2 r /2})\Big].
\end{align*}
For the last inequality, we  used that $e^{-x/2}+xe^{-x/2} + \indiq_{\{x<2\}} \leq C e^{-x/4}$  for all $x\geq 0$.
Since now $\intoi e^{-ar/2}(1-e^{-br/2})\frac{\dd r}{r}=\log(1+b/a)$ for all $a>0$, $b\geq 0$,
which can be checked by differentiating both side in $b$ and by using the value at $b=0$, we conclude that
\begin{align*}
\intoi\!\! \int_0^r \frac{F(r,s)}{rs}\dd s \dd r \leq& C\E\Big[\frac{\bM_1}{\bm_1} + \frac{\bM_2}{\bm_2}
\log \Big(1+\frac{\bm_2}{\bm_1}\Big)\Big]\leq C\E\Big[\frac{\bM_1}{\bm_1} + \frac{\bM_2}{\bm_2}
(1+|\log\bm_1|+|\log\bm_2|)\Big].
\end{align*}
This last quantity is finite by assumption and because $(\bm_1,\bM_1)$ and $(\bm_2,\bM_2)$
are independent and have the same law.
\end{proof}

Finally, we give the

\begin{proof}[Proof of Theorem \ref{thv}]
We assume that $f$ is continuous and bounded and 
\eqref{condv} for some $q>0$ and some $\theta \in (0,1]$ (and with $r_0=1$).
This implies that $\intrd f(x) \log^2f(x) \dd x<\infty$, because $m\leq v_d f \leq C$,
so that $\log^2 f \leq C(1+\log^2 m)$.
We thus can apply Proposition \ref{abb} and all the lemmas of the section.

\vip

We first check (ii) and thus assume that $\theta\in(0,1/2]$.
We then use  Lemma \ref{dec} with $\e=1$ and write $H_N=H_N^1+K^{1,1}_N+K^{2,1}_N+K^{3}_N$. 
But $H_N^1=0$ a.s., whence $\Var H_N\leq 3\Var K^{1,1}_N+3\Var K^{2,1}_N+3\Var K^{3}_N
\leq 3W^{11,1}_N+3W^{12,1}_N+3W^{21,1}_N+3W^{22,1}_N+3W^{31}_N+3W^{32}_N$ by Lemma \ref{expr2}.
By Lemmas \ref{lv1}, \ref{lv2}, \ref{lv3}, \ref{lv4} and \ref{lv5}, we find that
$\Var H_N \leq C N^{-1} + C N^{-1-\theta}+ C N^{-2\theta}$ for all $N\geq 2$,
whence $\Var H_N \leq CN^{-2\theta}$. The case $N=1$ is of course not an issue.

\vip

We next prove (i) and thus assume that $\theta \in (1/2,1]$.
For each $\e\in (0,1]$ and $N\geq 1/\e$, we write  $H_N=H_N^\e+K^{1,\e}_N+K^{2,\e}_N+K^{3}_N$ as in Lemma
\ref{dec}. We then infer from Proposition \ref{abb} that 
for each $\e\in(0,1]$, $\sqrt N(H_N^\e-\E[H_N^\e])$ goes in law to $\cN(0,\sigma_\e^2(f))$
as $N\to \infty$. We also know that $\lim_{\e\to 0} \sigma_\e^2(f)=\sigma^2(f)$ by Lemma \ref{lv0}.
To conclude that $\sqrt N(H_N-\E[H_N])$ goes in law to $\cN(0,\sigma^2(f))$, it only
remains to verify that $\lim_{\e\to 0} \sup_{N\geq 1/\e} N \Var (K^{1,\e}_N+K^{2,\e}_N+K^{3}_N)=0$.
But we have $\Var (K^{1,\e}_N+K^{2,\e}_N+K^{3}_N)\leq 3 (W^{11,\e}_N+W^{12,\e}_N+W^{21,\e}_N+W^{22,\e}_N+W^{31}_N+W^{32}_N)$,
see Lemma \ref{expr2}.
By Lemmas \ref{lv1}, \ref{lv2} and \ref{lv5}, it holds that
$\limsup_{\e\to 0} \sup_{N\geq 1/\e} N(W^{11,\e}_N+W^{12,\e}_N+W^{21,\e}_N+W^{22,\e}_N)\leq 0$.
We then infer from Lemma \ref{lv3} that $\sup_{N\geq 1/\e} NW^{31}_N\leq C \e^\theta$
and from Lemma \ref{lv4} that $\sup_{N\geq 1/\e} NW^{32}_N\leq C \e^{2\theta-1}$. Both
tend to $0$ as $\e\to 0$ because $\theta\in(1/2,1]$.
All in all, $\limsup_{\e\to 0} \sup_{N\geq 1/\e} N \Var (K^{1,\e}_N+K^{2,\e}_N+K^{3}_N)\leq 0$ as desired.
\end{proof}

\section{Estimation of the variance}\label{ee}

The goal of this section is to prove Proposition \ref{mp}. We have already done
most of the work.

\begin{proof}[Proof of Proposition \ref{mp}]
We assume that $f$ is bounded and continuous and satisfies \eqref{condestiv} for some $q>0$ and some $\theta>0$
(with $r_0=1$).
We write $V_N=A_N-B_N^2+\chi_d-\pi^2/6$, where
$$
A_N=\frac 1{N+1}\sum_{i=1}^{N+1} \log^2 Y^N_i \quad\hbox{and}\quad B_N=\frac 1{N+1}\sum_{i=1}^{N+1} \log Y^N_i.
$$
For $\e\in(0,1]$, we recall that $\log_\e y = \log(\e\lor y \land(1/\e))$ and define, for $N\geq 1/\e$,
$$
A_N^\e=\frac 1{N+1}\sum_{i=1}^{N+1} \log^2_\e Y^N_i \quad\hbox{and}\quad B_N^\e=\frac 1{N+1}\sum_{i=1}^{N+1} \log_\e Y^N_i.
$$

{\it Step 1.} For each fixed $\e\in (0,1]$, since $f:\rd\mapsto \rr_+$ and $\log_\e:[0,\infty)\mapsto \rr$ 
are bounded and continuous, we infer from \cite[Theorem 3.5]{bb} that
$\sup_{N\geq 1/\e} N (\Var(A^\e_N)+\Var(B^\e_N))<\infty$.

\vip

{\it Step 2.} Here we show that $Y^N_1$ goes in law to $\xi/(v_d\bff_1)$ as $N\to\infty$,
where $\xi\sim \Exp(1)$ is independent of $X_1$. For each $r>0$, we have
$\Pr(Y^N_1>r\,|\,X_1)=(1-\ba^N_1)^N=(1-f(B(X_1,(r/N)^{1/d})))^N$ by Lemma \ref{pdd}.
Since $f$ is continuous, $Nf(B(x,(r/N)^{1/d}))\to v_df(x)r$ as $N\to\infty$ for all $x\in\rd$.
Thus $\Pr(Y^N_1>r\,|\,X_1)$ a.s. tends to $\exp(-v_d \bff_1r)$ as $N\to \infty$. 
By dominated convergence, we conclude that
for each $r>0$, $\lim_{N\to\infty}\Pr(Y^N_1>r)=\E[\exp(-v_d\bff_1r)]$, which equals
$\Pr(\xi/(v_d\bff_1)>r)$.

\vip

{\it Step 3.} For each fixed $\e\in (0,1]$, since $\log_\e:[0,\infty)\mapsto \rr$ 
is bounded and continuous, we deduce from Step 2 that 
$\lim_{N\to \infty} \E[A^\e_N]=\lim_{N\to \infty} \E[\log^2_\e Y^N_1]=\E[\log^2_\e(\xi/(v_d\bff_1))]$
and that $\lim_{N\to \infty} \E[B^\e_N]=\lim_{N\to \infty} \E[\log_\e Y^N_1]=\E[\log_\e(\xi/(v_d\bff_1))]$.

\vip

{\it Step 4.} As seen in the first paragraph of the 
proof of Lemma \ref{lv0}, $\E[\log^2(\xi/(v_d\bff_1))]<\infty$
(we have $\intrd f(x)\log^2f(x)\dd x<\infty$ because $\intrd f(x)\log^2m(x)\dd x<\infty$
and $m\leq v_d f \leq C$ by assumption).
Thus $\lim_{\e\to 0}\E[\log^2_\e(\xi/(v_d\bff_1))]=\E[\log^2(\xi/(v_d\bff_1))]$
and $\lim_{\e\to 0}\E[\log_\e(\xi/(v_d\bff_1))]=\E[\log(\xi/(v_d\bff_1))]$
by dominated convergence.

\vip

{\it Step 5.} Here we verify that $\lim_{\e\to 0} \sup_{N\geq 1/\e} \E[|A_N^\e-A_N|]=0$.
It is checked similarly that $\lim_{\e\to 0} \sup_{N\geq 1/\e} \E[|B_N^\e-B_N|]=0$.
We recall that for all $\e>0$, all $N\geq 1/\e$, all $y\in(0,\infty)$,
$\log y = \log_\e y + \log [(y/\e)\land 1]+
\log [1\lor (\e y)\land (\e N)]+\log [(y/N)\lor 1]$ and (same formula with $\e=1$)
$\log y = \log [y\land 1]+\log [1\lor y \land N]+\log [(y/N)\lor 1]$,
see the proof of Lemma \ref{dec}. Starting from 
$|\log^2 y -\log^2_\e y |\leq 2|\log y| | \log y - \log_\e y|$, we end with
\begin{align*}
|\log^2 y -\log^2_\e y |
\leq  & 2\Big| \log [y\land 1]+
\log [1\lor y \land N]+\log [(y/N)\lor 1] 
\Big|\\
&\hskip2cm \times\Big| \log [(y/\e)\land 1]+
\log [1\lor (\e y)\land (\e N)]+\log [(y/N)\lor 1]  \Big|.
\end{align*}
Hence, by the Cauchy-Schwarz inequality, for $N\geq 1/\e$,
\begin{align*}
\E[|A_N^\e-A_N|]\leq& \E[|\log^2 Y^N_1 -\log^2_\e Y^N_1 |]\\
\leq&
C\E\Big[\log^2 [Y^N_1\land 1] + \log^2[1\lor Y^N_1 \land N] +\log^2[(Y^N_1/N)\lor 1] \Big]^{1/2}\\
&\times\E\Big[\log^2 [(Y^N_1/\e)\land 1]+
\log^2 [1\lor (\e Y^N_1)\land (\e N)]+\log^2 [(Y^N_1/N)\lor 1]  \Big]^{1/2}.
\end{align*}
With the notation of Lemma \ref{dec} (see also its proof), this precisely rewrites
\begin{align*}
\E[|A_N^\e-A_N|]\leq&C\Big( (N+1)W^{11,1}_N +(N+1)W^{21,1}_N+ (N+1)W^{31}_N\Big)^{1/2}\\
&\times \Big( (N+1)W^{11,\e}_N +(N+1)W^{21,\e}_N+ (N+1)W^{31}_N\Big)^{1/2}
\end{align*}
We then deduce from Lemmas \ref{lv1}, \ref{lv2} and \ref{lv3} that
$\lim_{\e\to 0}\sup_{N\geq 1/\e} \E[|A_N^\e-A_N|]=0$. We can apply these three lemmas thanks to
\eqref{condestiv} (with $r_0=1$) and because $f$ is bounded.

\vip

{\it Step 6.} Here we conclude that $\lim_{N\to\infty}\E[|A_N-\E[\log^2(\xi/(v_d\bff_1))]|]=0$
as $N\to\infty$. It is checked similarly that $\lim_{N\to\infty}\E[|B_N-\E[\log(\xi/(v_d\bff_1))]|]=0$.
We fix $\e\in (0,1]$ and write
\begin{align*}
\E[|A_N-\E[\log^2(\xi/(v_d\bff_1))]|]\leq& \E[|A_N-A_N^\e|]+\E[|A_N^\e-\E[\log^2_\e(\xi/(v_d\bff_1))]|]\\
&+|\E[\log^2_\e(\xi/(v_d\bff_1))]-\E[\log^2(\xi/(v_d\bff_1))]|.
\end{align*}
Taking first the limsup as $N \to \infty$ (so that the middle term of the RHS disappears by Steps 1 and 3)
and then the limsup as $\e\to 0$ (using Steps 4 and 5) completes the step.

\vip

{\it Step 7.} By Step 6, $V_N=A_N-B_N^2+\chi_d-\pi^2/6$ goes to 
$\Sigma=\Var(\log(\xi/(v_d\bff_1)))+\chi_d-\pi^2/6$ in probability.
But $\Var(\log(\xi/(v_d\bff_1)))=\intrd f(x)\log^2f(x)\dd x - (H(f))^2+\pi^2/6$,
see the first paragraph of the proof of Lemma \ref{lv0}. Thus
$\Sigma = \intrd f(x)\log^2f(x)\dd x - (H(f))^2+\chi_d=\sigma^2(f)$.
\end{proof}

\section{A tedious Taylor approximation}\label{dd}

Here we study in details how well $(1-f(B(x,(r/N)^{1/d}))^N$
approximates $\exp(-v_d f(x)r)$.

\begin{lem}\label{tedious}
Let  $\beta>0$, set $\rho=\min\{\beta,2\}$, 
$k=\max\{i\in\nn \; : \; i<\beta\}$ and $\ell=\max\{i\in\nn\;:\; 2i<\beta\}$.
Assume  that $\kappa=\sup_{x\in\rd}f(B(x,1))<1$, that $f\in \CC^\beta(\rd)$ and recall that
$M$ and $G_\beta$ were defined in \eqref{mM} and \eqref{gbeta}.
For each $N\geq 1$, 
consider $h_N:\rd\mapsto(0,\infty)$ such that
for all $x\in \rd$,
\begin{equation}\label{chn}
h_N(x)\leq \min\Big\{N, \frac{N^{\rho/(d+\rho)}}{(G_\beta(x))^{d/(d+\rho)}}, \frac{\sqrt N}{M(x)} \Big\}.
\end{equation}

(i) If $\beta \in (0,2]$, for all $N\geq 1$, all $x\in \rd$, all $r \in [0,h_N(x)]$,
\begin{align*}
[1-f(B(x,(r/N)^{1/d}))]^N = e^{-v_df(x)r}(1+R_N(x,r)),
\end{align*}
where $R_N$ satisfies, for some constant $C>0$, for all $N\geq 1$, all $x\in\rd$, all $r \in [0,h_N(x)]$,
\begin{align*}
|R_N(x,r)| \leq C \Big(\frac{r^2M^2(x)}{N} + \Big(\frac rN\Big)^{\beta/d} rG_\beta(x)\Big).
\end{align*}

(ii) If $\beta>2$, for all $N\geq 1$, all $x\in \rd$, all $r \in [0,h_N(x)]$,
\begin{align*}
[1-f(B(x,(r/N)^{1/d}))]^N = e^{-v_df(x)r}\Big(1+ \sum_{i=1}^\ell \frac{g_i(x,r)}{N^{2i/d}}  +S_N(x,r)\Big),
\end{align*}
where $S_N$ satisfies, for some constant $C>0$, for all $N\geq 1$, all $x\in\rd$, all $r\in [0,h_N(x)]$,
$$
|S_N(x,r)|\leq C \frac{r^2M^2(x)}{N} + C \Big(\frac rN\Big)^{\beta/d} 
[rG_\beta(x)+(rG_\beta(x))^{\beta/2}]
$$
and where the functions $g_1,\dots,g_\ell:\rd\times[0,\infty)\mapsto \rr$ (not depending on $N$)
satisfy, for some constant $C>0$, for all $x\in\rd$ and all $r\geq 0$,
\begin{align*}
|g_i(x,r)|\leq& C r^{2i/d}[rG_\beta(x)+(r G_\beta(x))^i].
\end{align*}
\end{lem}
 
\begin{proof} 
{\it Step 1.} For $\alpha=(\alpha_1,\dots,\alpha_d)\in \nn^d$ a multi-index,
we use the standard notation $|\alpha|=\sum_{i=1}^d \alpha_i$, $\alpha ! =\alpha_1!\dots \alpha_d!$ and
$h^\alpha=h_1^{\alpha_1}\dots h_d^{\alpha_d}$ for $h\in\rd$. 
Using the Taylor formula, we write, for $y\in B(x,1)$, with the convention that $\sum_1^0=0$ if $k=0$
(i.e. $\beta \in (0,1]$),
\begin{align*}
f(y)=f(x)+\sum_{|\alpha|=1}^{k} \frac 1 {\alpha !} \partial_\alpha f(x) (y-x)^\alpha + \Delta_\beta(x,y),
\end{align*}
with $|\Delta_\beta(x,y)| \leq C |x-y|^\beta G_\beta(x)$. 
The function $G_\beta$ was precisely designed for that purpose.

\vip

For $\e\in(0,1]$, we integrate the above equality on $B(x,\e)$:
\begin{align*}
f(B(x,\e))=v_d\e^df(x)+\sum_{|\alpha|=1}^{k} \frac 1 {\alpha !} \partial_\alpha f(x) \int_{B(x,\e)}(y-x)^\alpha \dd y + 
\int_{B(x,\e)}\Delta_\beta(x,y)\dd y.
\end{align*}
But one easily checks that $\int_{B(x,\e)}(y-x)^\alpha \dd y=0$ if $|\alpha|$ is odd, while
$\int_{B(x,\e)}(y-x)^\alpha \dd y=c_\alpha \e^{|\alpha|+d}$ when $|\alpha|$ is even, with
$c_\alpha=\int_{B(0,1)}y^\alpha \dd y$. We thus may write, using the convention that $\sum_1^0=0$ when $\ell=0$
(i.e. $\beta \in (0,2]$),
\begin{equation*}
f(B(x,\e))=v_d \e^d f(x)+\sum_{j=1}^{\ell} \e^{2j+d}\delta_{2j}f(x) + \e^{d+\beta}\Gamma_\beta(x,\e),
\end{equation*}
with $|\Gamma_\beta(x,\e)|=\e^{-d-\beta}|\int_{B(x,\e)}\Delta_\beta(x,y)\dd y|\leq C  G_\beta(x)$
and $\delta_{2j}f(x)=\sum_{|\alpha|=2j} \frac 
{c_\alpha}{\alpha !} \partial_\alpha f(x)$, which satisfies $|\delta_{2j}f(x)|\leq C G_\beta(x)$
(because $2j \leq 2\ell < \beta$).

\vip

We conclude that for all $r\in[0,N]$, all $x\in\rd$,
\begin{equation}\label{ettoc}
Nf(B(x,(r/N)^{1/d}))=v_d f(x)r+r \sum_{j=1}^{\ell} \Big(\frac rN\Big)^{2j/d}\delta_{2j}f(x) 
+ r\Big(\frac rN\Big)^{\beta/d}G_\beta(x) \theta_{N,1}(x,r),
\end{equation}
for some uniformly bounded $\theta_{N,1}$, that is, 
$\sup_{N\geq 1}\sup_{x\in\rd}\sup_{r\in[0,N]} |\theta_{N,1}(x,r)|<\infty$. Also, one easily verifies
that there is a constant $C>0$ such that for all $N\geq 1$, $x\in\rd$, $r\in[0,h_N(x)]$,
\begin{equation}\label{rehn}
r \sum_{j=1}^{\ell} \Big(\frac rN\Big)^{2j/d}|\delta_{2j}f(x)|
+ r\Big(\frac rN\Big)^{\beta/d}G_\beta(x)\leq C.
\end{equation}
It suffices to use that $|\delta_{2j}f(x)|\leq CG_\beta(x)$, that
$(r/N)^{2j/d}+(r/N)^{\beta/d} \leq 2 (r/N)^{\rho/d}$ for all $j=1,\dots,\ell$ (use that $h_N(x)\leq N$ whence
$r/N\leq 1$ and recall that $\rho=\min\{\beta,2\}$)
and that $r (r/N)^{\rho/d}G_\beta(x)\leq 1$ because 
$r\leq h_N(x) \leq N^{\rho/(d+\rho)}/(G_\beta(x))^{d/(d+\rho)}$.
\vip

{\it Step 2.} For $0\leq r \leq N$, we have $f(B(x,(r/N)^{1/d}))\leq \kappa<1$.
Hence 
$$
\log [1-f(B(x,(r/N)^{1/d}))]=-f(B(x,(r/N)^{1/d})) + [f(B(x,(r/N)^{1/d}))]^2\theta_{N,2}(x,r),
$$
for some uniformly bounded function $\theta_{N,2}$.
And we have $[f(B(x,(r/N)^{1/d}))]^2 \leq v_d^2N^{-2}r^2 M^2(x)$ by definition of $M$.
As a consequence, for some new uniformly bounded function $\theta_{N,3}$,
\begin{equation}\label{tropcool}
[1-f(B(x,(r/N)^{1/d}))]^N=\exp\Big(-Nf(B(x,(r/N)^{1/d})) + N^{-1}r^2 M^2(x)\theta_{N,3}(x,r)\Big).
\end{equation}
Now since $h_N(x)\leq\sqrt N / M(x)$, it holds that for all $N\geq 1$, $x\in\rd$, $r\in[0,h_N(x)]$,
\begin{equation}\label{rehn2}
N^{-1}r^2 M^2(x) \leq 1.
\end{equation}
Combining \eqref{ettoc} and \eqref{tropcool} gives us, for $x\in\rd$ and $r\in [0,h_N(x)]$,
\begin{align*}
[1-f(B(x,(r/N)^{1/d}))]^N=e^{-v_df(x)r}\exp\Big(I_N(x,r)+J_N(x,r)\Big),
\end{align*}
where 
\begin{align*}
I_N(x,r)=&-r \sum_{j=1}^{\ell} \Big(\frac rN\Big)^{2j/d}\delta_{2j}f(x),\\
J_N(x,r)=&- r\Big(\frac rN\Big)^{\beta/d}G_\beta(x) \theta_{N,1}(x,r)
+  \frac {r^2 M^2(x)}N\theta_{N,3}(x,r).
\end{align*}
Using that $I_N$ and $J_N$ are uniformly bounded
(for $r\in[0,h_N(x)]$) by \eqref{rehn} and \eqref{rehn2}, we may write
\begin{align}\label{pasmal}
[1-f(B(x,(r/N)^{1/d}))]^N=e^{-v_df(x)r}\Big(1+\sum_{i=1}^\ell\frac 1 {i!}(I_N(x,r))^i +R_N(x,r)\Big),
\end{align}
with (since $\ell+1\geq \beta/2$ by definition of $\ell$)
\begin{align*}
|R_N(x,r)|\leq C (|I_N(x,r)|^{\ell+1}+ |J_N(x,r)|)\leq C (|I_N(x,r)|^{\beta/2}+ |J_N(x,r)|).
\end{align*}

{\it Step 3.} If $\beta\in(0,2]$, then $\ell=0$ and thus $I_N=0$.
Using \eqref{pasmal} and noting that 
$$
|R_N(x,r)|\leq C|J_N(x,r)|\leq C r \Big(\frac r N\Big)^{\beta/d}G_\beta(x)
+ C \frac {r^2M^2(x)} N
$$
for all $x\in\rd$ and all $r\in [0,h_N(x)]$ completes the proof of (i).

\vip

{\it Step 4.} We now suppose that $\beta>2$, whence $\ell\geq 1$ and $\rho=2$. 
First, we have $|I_N(x,r)|\leq C r (r/N)^{2/d} G_\beta(x)$ for all $r\in[0,h_N(x)]$, 
because $|\delta_{2j}f(x)| \leq CG_\beta(x)$ and because $r/N\leq 1$.
Thus
\begin{align}\label{tbua}
|R_N(x,r)|\leq  C \Big(\frac rN\Big)^{\beta/d}(r G_\beta(x))^{\beta/2}
+C \Big(\frac rN\Big)^{\beta/d}rG_\beta(x)
+ C \frac {r^2 M^2(x)}N.
\end{align}
Using next the
multinomial theorem, we find (here $i_1,\dots,i_\ell$ are non-negative integers)
\begin{align*}
\sum_{i=1}^\ell\frac 1 {i!}(I_N(x,r))^i=
& \sum_{i=1}^\ell (-r)^i \sum_{i_1+\cdots+i_\ell=i}\Big(\frac r N\Big)^{2(i_1+2i_2+\cdots+\ell i_\ell)/d}
\frac{(\delta_2f(x))^{i_1}(\delta_4f(x))^{i_2}\cdots(\delta_{2\ell}f(x))^{i_\ell}}{i_1!\cdots i_\ell!}\\
=& \sum_{m=1}^\ell \frac{g_m(x,r)}{N^{2m/d}} + T_N(x,r),
\end{align*}
where
\begin{align*}
g_m(x,r)=r^{2m/d}\sum_{i_1+2i_2+\cdots+\ell i_\ell=m}(-r)^{i_1+\cdots+i_\ell}
\frac{(\delta_2f(x))^{i_1}(\delta_4f(x))^{i_2}\cdots(\delta_{2\ell}f(x))^{i_\ell}}{i_1!\cdots i_\ell!}
\end{align*}
is well-defined on $\rd\times[0,\infty)$ and where
\begin{align*}
T_N(x,r)=\sum_{i=1}^\ell (-r)^i \sum_{\begin{array}{c}\scriptstyle i_1+\cdots+i_\ell=i\\ \scriptstyle
i_1+2i_2+\cdots+\ell i_\ell>\ell\end{array}}\Big(\frac r N\Big)^{2(i_1+2i_2+\cdots+\ell i_\ell)/d}
\frac{(\delta_2f(x))^{i_1}(\delta_4f(x))^{i_2}\cdots(\delta_{2\ell}f(x))^{i_\ell}}{i_1!\cdots i_\ell!}.
\end{align*}
Recalling \eqref{pasmal}, we have checked that for all $x\in\rd$, all $r\in[0,h_N(x)]$,
\begin{align*}
(1-f(B(x,(r/N)^{1/d}))^N = e^{-v_df(x)r}\Big(1+ \sum_{m=1}^\ell \frac{g_m(x,r)}{N^{2m/d}}  +S_N(x,r)\Big),
\end{align*}
with $S_N(x,r)=T_N(x,r)+R_N(x,r)$.

\vip

Since $|\delta_{2j}f(x)|\leq CG_\beta(x)$,
we deduce that for all $m=1,\dots,\ell$, all $x\in\rd$, all $r\geq 0$,
\begin{align*}
|g_m(x,r)|\leq C r^{2m/d} \sum_{i_1+2i_2+\cdots+\ell i_\ell=m} (rG_\beta(x))^{i_1+\cdots+i_\ell}\leq
C r^{2m/d}[rG_\beta(x) +(rG_\beta(x))^m],
\end{align*}
because $i_1+2i_2+\cdots+\ell i_\ell=m$ implies $1\leq i_1+\cdots+i_\ell \leq m$.

\vip

Similarly, using that $r\leq N$,
\begin{align*}
|T_N(x,r)|\leq C \Big(\frac r N \Big)^{2(\ell+1)/d}[r G_\beta(x) +(rG_\beta(x))^\ell]
\leq C \Big(\frac r N \Big)^{\beta /d}[rG_\beta(x) +(rG_\beta(x))^{\beta/2}]
\end{align*}
because $2(\ell+1)\geq \beta$ and $1\leq \ell<\beta/2$. Recalling \eqref{tbua}, we find that
\begin{align*}
|S_N(x,r)|\leq  C \Big(\frac r N \Big)^{\beta /d}[rG_\beta(x) +(rG_\beta(x))^{\beta/2}]
+ C \frac {r^2 M^2(x)}N
\end{align*}
as desired.
\end{proof}

\section{Bias}\label{bb}

The whole section is devoted to the proof of Theorems \ref{thb1} and \ref{thb2}.
We first provide an integral expression of the bias.

\begin{lem}\label{expr}
We have
\begin{align*}
\E[H_N]=&H(f)+\E\Big[\int_0^\infty \Big((1-\ba^N_1(r))^N-e^{-v_d\bff_1r}\Big)\frac{\dd r}r \Big].
\end{align*}
\end{lem}

\begin{proof}
Since $\log z=\intoi (\indiq_{\{r<z\}}-\indiq_{\{r<1\}})\frac{\dd r}r$ for $z>0$, for any
positive random variable $Z$,
\begin{align*}
\E[\log Z]= \intoi [\Pr(Z> r)-\indiq_{\{r<1\}}]  \frac{\dd r}r.
\end{align*}
Recalling that $\E[H_N]=\gamma +\log v_d +\E[\log Y^N_1]$, see \eqref{mh}, 
we deduce from Lemma \ref{pdd} that
\begin{equation}\label{j1}
\E[H_N]=\gamma+\log v_d+ \E\Big[\int_0^\infty [(1-\ba^N_1(r))^N-\indiq_{\{r<1\}}]\frac{\dd r}r\Big].
\end{equation}
Also, for $\xi\sim\Exp(1)$ independent of $X_1$, so that $\Pr(\xi/(v_d\bff_1)>r\,|\,X_1)=\exp(-v_d\bff_1r)$,
\begin{align*}
\E\Big[\log \frac{\xi}{v_d\bff_1}\Big] = 
\E\Big[\intoi [e^{-v_d\bff_1r}  -\indiq_{\{r<1\}} ]\frac{\dd r}r\Big].
\end{align*}
But $\E[\log (\xi/(v_d\bff_1))] = \E[\log \xi] - \log v_d - \E[\log \bff_1]
=-\gamma - \log v_d + H(f)$, so that we have $H(f)=\gamma +\log v_d +
\E[\intoi [e^{-v_d\bff_1r}  -\indiq_{\{r<1\}} ]\frac{\dd r}r]$.
Recalling \eqref{j1}, one easily concludes.
\end{proof}

From now on, we fix $\beta>0$ and we set $\rho=\min\{\beta,2\}$ and
$\ell=\max\{i\in\nn\;:\; 2i<\beta\}$. We assume that $f\in \CC^\beta(\rd)$ and recall that
$M$ and $G_\beta$ were defined in \eqref{mM} and \eqref{gbeta}. We assume that $\kappa=\sup_{\rd}f(B(x,1))<1$.
We put $R=M+G_\beta$ and introduce, for each $N\geq1$, 
the function $h_N:\rd\mapsto [0,N]$ defined by
\begin{align*}
h_N(x)=\min\Big\{\frac{2\log N}{m(x)},N, \frac{N^{\rho/(d+\rho)}}{(R(x))^{d/(d+\rho)}},
\frac{\sqrt N}{R(x)} \Big\},
\end{align*}
which of course satisfies \eqref{chn}. We also introduce
the shortened notation
\begin{align*}
\bR_1=R(X_1)\quad\hbox{and}\quad\bh_{N1}=h_N(X_1).
\end{align*}
We observe that $\Omega =\cup_{i=1}^4 \Omega_N^i$, where
\begin{align*}
\Omega_N^1\!=\!\Big\{\bh_{N1}= \frac{2\log N}{\bm_1}\Big\},  \quad
\Omega_N^2\!=\!\Big\{\bh_{N1}=N\Big\},\quad
\Omega_N^3\!=\!\Big\{\bh_{N1}=\frac{N^{\rho/(d+\rho)}}{\bR_1^{d/(d+\rho)}}\Big\}, \quad
\Omega_N^4\!=\!\Big\{\bh_{N1}=\frac{\sqrt N}{\bR_1}\Big\}.
\end{align*}
We infer from Lemma \ref{expr} that $\E[H_N]-H(f)= B^N_1+B^N_2+B^N_3+B^N_4$, where
\begin{gather*}
B^N_1=  \E\Big[\int_0^{\bh_{N1}}\Big((1-\ba^N_1(r))^N-e^{-v_d\bff_1r}\Big)\frac{\dd r}{r}\Big],\\
B^N_2= -\E\Big[\int_{\bh_{N1}}^\infty e^{-v_d\bff_1r}\frac{\dd r}{r}\Big],\quad
B^N_3= \E\Big[\int_{\bh_{N1}}^N (1-\ba^N_1(r))^N\frac{\dd r}{r}\Big],\quad
B^N_4= \E\Big[\int_N^{\infty} (1-\ba^N_1(r))^N\frac{\dd r}{r}\Big].
\end{gather*}

The two terms $B^N_2$ and $B^N_3$ can be studied together.

\begin{lem}\label{lb1}
Assume that $\intrd [R^{2\theta}(x)/m^{2\theta}(x)+ R^{\theta d/\rho}(x)/m^{\theta (d+\rho)/\rho}(x)]
f(x)\dd x<\infty$
for some $\theta\in(0,1]$.
Then for all $N\geq 1$,
$|B^N_2|+|B^N_3|\leq C N^{-\theta}$.
\end{lem}

\begin{proof}
First, $(1-\ba^N_1(r))^N\leq  \exp(-\bm_1 r)$ for all $r\in[0,N]$ by Lemma \ref{tbust}-(ii).
Since furthermore $m(x) \leq v_d f(x)$ by definition of $m$,
\begin{align*}
|B^N_2|+|B^N_3|\leq 2 \E\Big[\int_{\bh_{N1}}^\infty e^{-\bm_1r}\frac{\dd r}{r}\Big]
=2\E\Big[\int_{\bh_{N1}\bm_1}^\infty e^{-u}\frac{\dd u}{u}\Big]\leq 2\E[\Phi(\bh_{N1}\bm_1)],
\end{align*}
where $\Phi(x)=\exp(-x)\indiq_{\{x\geq 1\}}+[1+\log(1/x)]\indiq_{\{x <1\}}$.

\vip

On $\Omega_N^1$, $\bh_{N1}\bm_1 =2\log N$, whence $\Phi(\bh_{N1}\bm_1)\leq C N^{-2}$
and $\E[\Phi(\bh_{N1}\bm_1)\indiq_{\Omega_N^1}]\leq C N^{-2}$.

\vip

On  $\Omega_N^2$, we have $\bh_{N1}\bm_1 =N \bm_1$. Furthermore, there is a constant $C>0$ such that 
$\Phi(x)\leq C x^{-\theta}$ for all $x>0$. Thus, recalling that $m(x)\leq v_d f(x)\leq v_d R(x)$,
$$
\E[\Phi(\bh_{N1}\bm_1)\indiq_{\Omega_N^2}]\leq \frac C {N^{\theta}}\E\Big[\frac 1 {\bm_1^\theta}\Big]
=\frac C {N^{\theta}}\E\Big[\frac {\bR_1^{\theta d/\rho}} {\bm_1^\theta \bR_1^{\theta d/\rho}}\Big]
\leq \frac C {N^{\theta}}\E\Big[\frac {\bR_1^{\theta d/\rho}} {\bm_1^{\theta (d+\rho)/\rho }}\Big]
\leq \frac C {N^{\theta}}.
$$

On $\Omega_N^3$, we have $\bh_{N1}\bm_1 =N^{\rho/(d+\rho)} \bm_1 /\bR_1^{d/(d+\rho)}$
and there is $C>0$ such that $\Phi(x)\leq C x^{-\theta(d+\rho)/\rho}$ for all $x>0$.
Thus
\begin{align*}
\E[\Phi(\bh_{N1}\bm_1)\indiq_{\Omega_N^3}]\leq 
\frac C{N^\theta}\E\Big[\frac{\bR_1^{\theta d/\rho}}{\bm_1^{\theta(d+\rho)/\rho}}\Big]\leq \frac C {N^{\theta}}.
\end{align*}

On $\Omega_N^4$, $\bh_{N1}\bm_1 =\sqrt N \bm_1/\bR_1$ and there is $C>0$ such that
$\Phi(x)\leq C x^{-2\theta}$ for all $x>0$, whence
\begin{align*}
\E[\Phi(\bh_{N1}\bm_1)\indiq_{\Omega_N^4}]\leq \frac C{N^\theta} \E\Big[\frac{\bR_1^{2\theta}}{\bm_1^{2\theta}}\Big]
\leq  \frac C{N^\theta}.
\end{align*}
This completes the proof.
\end{proof}

\begin{lem}\label{lb2}
Assume that $\intrd [|x|^q+ \log(2+|x|)[f(B(x,1))]^{-\theta}]f(x)\dd x<\infty$ 
for some $q>0$ and some $\theta \in(0,1]$. Then for all $N\geq 1$, $|B^N_4|\leq  C N^{-\theta}$.
\end{lem}

\begin{proof}
Recall that $B^N_4=\E[\int_N^{\infty} (1-\ba^N_1(r))^N\frac{\dd r}{r}]=
\E[\int_1^{\infty} (1-\ba^N_1(Nu))^N\frac{\dd u}{u}]$. We
introduce $g(x)=1\lor\E[|X_1-x|^q]$ and $\bg_1=g(X_1)$ as usual.
We write $B^N_4=I_N+J_N$,
where  
$$
I_N=\E\Big[\int_1^{(2\bg_1)^{d/q}} (1-\ba^N_1(Nu))^N\frac{\dd u}{u}\Big]
\quad \hbox{and}\quad J_N=\E\Big[\int_{(2\bg_1)^{d/q}}^\infty (1-\ba^N_1(Nu))^N\frac{\dd u}{u}\Big].
$$

First, there is $C>0$ such that for all $x\in(0,1]$, all $N\geq 1$,
$(1-x)^N\leq e^{-Nx}\leq C (Nx)^{-\theta}$, whence $(1-\ba^N_1(Nu))^N\leq 
C /(N \ba^N_1(Nu))^{\theta}\leq C /(N f(B(X_1,1)))^{\theta}$ for all $u\geq 1$. Hence
$$
I_N \leq \frac C{N^{\theta}} \E\Big[\frac 1 {(f(B(X_1,1)))^{\theta}}\int_1^{(2\bg_1)^{d/q}}\frac {\dd r}r\Big]
\leq \frac C{N^{\theta}} \E\Big[\frac{\log (2+|X_1|)}{(f(B(X_1,1)))^{\theta}}\Big]
\leq \frac C{N^{\theta}}.
$$
We used that $|\log (2\bg_1)^{d/q}|\leq C\log (2+|X_1|)$ because $1\leq g(x)\leq C(1+|x|^q)$.

\vip

Next, since $1-\ba^N_1(Nu)\leq \bg_1u^{-q/d}$ by Lemma \ref{tbust}-(iv),
$$
J_N \leq \E\Big[ \int_{(2\bg_1)^{d/q}}^\infty \Big(\frac{\bg_1}{u^{q/d}}\Big)^N \frac {\dd u}u\Big]
\leq \Big(\frac12\Big)^{N-1}\E\Big[\bg_1 \int_1^\infty \frac {\dd u}{u^{1+q/d}}\Big]\leq \frac C {2^{N-1}}
$$
because $\E[\bg_1]<\infty$. This completes the proof.
\end{proof}

We next treat $B^N_1$ when $\beta \in (0,2]$.

\begin{lem}\label{lb4}
Assume $\beta \in (0,2]$ and $\intrd [R^{2\theta}(x)/f^{2\theta}(x)+R^{\theta d/\beta}(x)/f^{\theta(\beta+d)/\beta}(x)]
f(x)\dd x<\infty$ for some $\theta \in [0,1\land(\beta/d)]$.
Then for all $N\geq 1$, $|B^N_1|\leq  CN^{-\theta}$.
\end{lem}

\begin{proof}
Thanks to Lemma \ref{tedious}-(i), we have 
$B^N_1=\E[\int_0^{\bh_{N1}} e^{-v_d\bff_1r}R_N(X_1,r)  \frac{\dd r}r]$, with furthermore
$|R_N(X_1,r)|\leq C (N^{-1}r^2\bR_1^2+ N^{-\beta/d} r^{1+\beta/d}  \bR_1)$.
Hence
$|B^N_1|\leq C(I_N+J_N)$, with 
\begin{align*}
I_N=\frac 1N\E\Big[\bR_1^2\int_0^{\bh_{N1}} e^{-v_d\bff_1r}r\dd r\Big]
\quad \hbox{and}\quad J_N=\frac 1{N^{\beta/d}}\E\Big[\bR_1\int_0^{\bh_{N1}} e^{-v_d\bff_1r} r^{\beta/d}  \dd r\Big].
\end{align*}

First, for $a,b \geq0$, we have 
$\int_0^a e^{-br}r\dd r\leq 2(\min\{b^{-1},a\})^2$ because $\intoi e^{-br}r\dd r=2b^{-2}$
and $\int_0^a r\dd r=a^2/2$. Since now $\theta \in [0,1]$, we deduce that 
$\int_0^a e^{-br}r\dd r\leq 2 b^{-2\theta}a^{2-2\theta}$. Hence
$$
I_N \leq \frac 2 N \E\Big[\frac{\bR_1^2 \bh_{N1}^{2-2\theta}}{(v_d\bff_1)^{2\theta}} \Big].
$$
Since now $h_N(x)\leq \sqrt{N}/R(x)$, we end with $I_N \leq C N^{-\theta} \E[\bR_1^{2\theta} \bff_1^{-2\theta}]
\leq CN^{-\theta}$.

\vip

Next, we observe that $\int_0^a e^{-br}r^{\beta/d}\dd r\leq C (\min\{b^{-1},a\})^{1+\beta/d}\leq 
C(b^{-\theta d/\beta}a^{1-\theta d/\beta} )^{1+\beta/d}$
because $\theta d/\beta \in[0,1]$.
Hence 
$$
J_N \leq \frac C {N^{\beta/d}}\E\Big[\frac{\bR_1 \bh_{N1}^{(1-\theta d/\beta)(1+\beta/d)} }
{(v_d\bff_1)^{(\theta d/\beta)(1+\beta/d)}}\Big].
$$
But we have $\beta\leq 2$ so that $\rho=\beta$ and thus 
$\bh_{N1}\leq N^{\beta/(\beta+d)}/\bR_1^{d/(\beta+d)}$. This precisely gives
$J_N \leq C N^{-\theta}\E[\bR_1^{\theta d/\beta} / \bff_1^{\theta(\beta+d)/\beta}]\leq C N^{-\theta}$.
\end{proof}

We finally study $B^N_1$ when $\beta>2$.

\begin{lem}\label{lb5}
Assume $\beta\in (2,d]$, recall that $\ell=\max\{i\in\nn\;:\;2i<\beta\}$ and take for granted that 
$\intrd [R^{2\beta/d}(x)/f^{2\beta/d}(x) + R^{\beta/2}(x)/f^{\beta/d+\beta/2}(x)]f(x)\dd x<\infty$.
There are some constants
$\lambda_1,\dots,\lambda_\ell \in \rr$ such that for all $N\geq 1$,
\begin{align*}
\Big|B^N_1-\sum_{i=1}^\ell \frac{\lambda_i}{N^{2i/d}}\Big|\leq \frac C{N^{\beta/d}}.
\end{align*}
\end{lem}

\begin{proof}
Using the notation of Lemma \ref{tedious}-(ii), we write
$$
B^N_1=\E\Big[\int_0^{\bh_{N1}} e^{-v_d\bff_1r}\Big(\sum_{i=1}^\ell \frac{g_i(X_1,r)}{N^{2i/d}}
+S_N(X_1,r)\Big)\frac{\dd r}r\Big]=
\sum_{i=1}^\ell \frac{\lambda_i}{N^{2i/d}} + \sum_{i=1}^\ell \Delta_N^i + \e_N, 
$$
where
\begin{gather*}
\lambda_i=\E\Big[\intoi e^{-v_d\bff_1r}g_i(X_1,r) \frac{\dd r}r\Big], \quad
\Delta_N^i=-\frac1{N^{2i/d}}\E\Big[\int_{\bh_{N1}}^\infty e^{-v_d\bff_1r}g_i(X_1,r) \frac{\dd r}r\Big],\\
\e_N=\E\Big[\int_0^{\bh_{N1}} e^{-v_d\bff_1r}S_N(X_1,r)  \frac{\dd r}r\Big].
\end{gather*}
It remains to prove that $\lambda_1,\dots,\lambda_\ell$ are well-defined and finite,
that $\sum_{i=1}^\ell |\Delta_N^i|\leq CN^{-\beta/d}$ and that $|\e_N|\leq  CN^{-\beta/d}$, which we
do successively in the three following steps.

\vip

{\it Step 1.} Recalling that $|g_i(x,r)|\leq C r^{2i/d}[r R(x)+(r R(x))^i]\leq C r^{2i/d}[1+(r R(x))^i]$,
we see that $\lambda_i$ is well defined for all $i=1,\dots,\ell$ because
\begin{align*}
\E\Big[\intoi \!\! e^{-v_d\bff_1r}r^{2i/d}[1+(r\bR_1)^i] \frac{\dd r}r\Big]
=\E\Big[\intoi \!\! e^{-v_du}\Big(\frac{u} {\bff_1}\Big)^{2i/d}\Big[1
+\Big(u\frac{\bR_1}{\bff_1}\Big)^i\Big] 
\frac{\dd u}u\Big]
\leq  C \E\Big[ \frac{\bR_1^i}{\bff^{2i/d+i}_1}\Big]
\end{align*}
(we used that $\bR_1\geq \bff_1$). This is finite 
because $i\leq \ell\leq \beta/2$ and because $\E[\bR_1^{\beta/2}/\bff^{\beta/d+\beta/2}_1]<\infty$.

\vip

{\it Step 2.} Since $|S_N(x,r)|\leq C r^2R^2(x)/N + C (r/N)^{\beta/d} 
[rR(x)+(rR(x))^{\beta/2}]$  for all $r\in[0,h_N(x)]$ and since
$rR(x)+(rR(x))^{\beta/2}\leq 1+2(rR(x))^{\beta/2}$ (recall that $\beta>2$), we have
$\e_N\leq\e_N^1+\e_N^2$, with
$$
\e_N^1=\frac CN\E\Big[\bR_1^2\int_0^{\bh_{N1}} e^{-v_d\bff_1r}r\dd r\Big]\quad\hbox{and}\quad
\e_N^2=\frac C{N^{\beta/2}}\E\Big[\int_0^{\bh_{N1}} e^{-v_d\bff_1r} r^{\beta/d} 
[ 1 + (r\bR_1)^{\beta/2}] 
\frac{\dd r}r\Big].
$$
First,
\begin{align*}
\e_N^2 \leq &\frac C{N^{\beta/2}}
\E\Big[\intoi e^{-v_du} \Big(\frac u {\bff_1}\Big)^{\beta/d} 
\Big[ 1 + \Big(\frac{u\bR_1}{\bff_1}\Big)^{\beta/2}\Big]
\Big) \frac{\dd u}u\Big]
\leq \frac C {N^{\beta/d}}\E\Big[\frac{\bR^{\beta/2}_1}{\bff^{\beta/d+\beta/2}_1} \Big] \leq\frac C {N^{\beta/d}} .
\end{align*}
Next, since $\int_0^a e^{-br}r\dd r\leq 2(\min\{b^{-1},a\})^2\leq 2 b^{-2\beta/d}a^{2-2\beta/d}$ for $a,b>0$
(recall that $\beta/d\in (0,1]$) and since $\bh_{N1}\leq \sqrt N /\bR_1$,
$$
\e_N^1 \leq \frac C N \E\Big[\bR_1^2 \frac {\bh_{N1}^{2-2\beta/d}}{\bff_1^{2\beta/d}}\Big]
\leq \frac C {N^{\beta/d}}\E\Big[\frac {\bR_1^{2\beta/d}}{\bff_1^{2\beta/d}}\Big]\leq\frac C {N^{\beta/d}} .
$$

{\it Step 3.} 
For $i=1,\dots,\ell$, using that $|g_i(x,r)|\leq C r^{2i/d}[1+(rR(x))^i]$,
\begin{align*}
|\Delta_N^i|\leq & \frac C  {N^{2i/d}}\E\Big[\int_{\bh_{N1}}^\infty e^{-v_d\bff_1r}  r^{2i/d}[1+
(r\bR_1)^i]\frac{\dd r}r\Big]
\leq  \frac C{N^{2i/d}}\E\Big[\frac{\bR_1^i}{\bff_1^{2i/d+i}}
\int_{\bh_{N1}\bff_1}^\infty e^{-v_d u}  u^{2i/d}[1+u^i] \frac{\dd u}u\Big].
\end{align*}
We used that $f\leq R$ for the last inequality.
Since now  $e^{-v_d u}  u^{2i/d}[1+u^i] \leq C e^{-v_d u /2}$, we find
$$
|\Delta_N^i|\leq \frac C{N^{2i/d}}  \E\Big[\frac{\bR_1^i}{\bff_1^{2i/d+i}} e^{-v_d \bh_{N1}\bff_1/2}\Big]
=C  \E\Big[\Big(\frac{\bR_1}{N^{2/d}\bff_1^{2/d+1}}\Big)^i e^{-v_d \bh_{N1}\bff_1/2}\Big].
$$
It follows that $\sum_{i=1}^\ell|\Delta_N^i| \leq C \E[\eta_N]$, where 
\begin{align*}
\eta_N=\Big[\frac{\bR_1}{N^{2/d}\bff_1^{2/d+1}}+\Big(\frac{\bR_1}{N^{2/d}\bff_1^{2/d+1}}\Big)^\ell
\Big]e^{-v_d \bh_{N1}\bff_1/2}.
\end{align*}

On $\Omega_N^1$, we have $v_d \bh_{N1}\bff_1/2= (v_d\bff_1/\bm_1) \log N \geq \log N$, so that
\begin{align*}
\E[\eta_N\indiq_{\Omega_N^1}] \leq \frac 1 N \E\Big[\frac{\bR_1}{N^{2/d}\bff_1^{2/d+1}}
+\Big(\frac{\bR_1}{N^{2/d}\bff_1^{2/d+1}}\Big)^\ell\Big]
\leq \frac C N \E\Big[ 1+ \frac{\bR_1^{\beta/2}}{\bff_1^{\beta/d+\beta/2}}\Big]\leq \frac CN.
\end{align*}
We used that $1\leq \ell \leq \beta/2$.

\vip

On $\Omega_N^2$, we have $\bh_{N1}=N$. Furthermore, there is a constant 
$C>0$ such that for all $x>0$, $e^{-v_d x/2}\leq C x^{(2-\beta)/d}$ and $e^{-v_d x/2}\leq C x^{(2\ell-\beta)/d}$
(recall that $\beta>2\ell\geq 2$). Hence
$$
\eta_N\indiq_{\Omega_N^2}\leq 
\Big[\frac{\bR_1}{N^{2/d}\bff_1^{2/d+1}}+\Big(\frac{\bR_1}{N^{2/d}\bff_1^{2/d+1}}\Big)^\ell
\Big]e^{-v_d N\bff_1/2}\leq \frac C{N^{\beta/d}\bff_1^{\beta/d}} 
\Big(\frac{\bR_1}{\bff_1}+\Big(\frac{\bR_1}{\bff_1}\Big)^\ell \Big).
$$
Since $\bff_1\leq v_d\bM_1\leq v_d\bR_1$ and since $1\leq \ell \leq \beta/2$, we conclude that
$$
\E[\eta_N\indiq_{\Omega_N^2}]\leq  
\frac C{N^{\beta/d}} \E\Big[\frac 1{\bff_1^{\beta/d}}\Big(\frac{\bR_1}{\bff_1}\Big)^{\beta/2}\Big]=
\frac C{N^{\beta/d}} 
\E\Big[\frac{\bR_1^{\beta/2}}{\bff_1^{\beta/d+\beta/2}}\Big]
\leq \frac C{N^{\beta/d}}.
$$

On $\Omega_N^3$, we have $\bh_{N1}\bff_1=N^{2/(d+2)}\bff_1/\bR_1^{d/(d+2)}$ 
(recall that $\rho=2$ because $\beta>2$).
Moreover, there is  $C>0$ such that $e^{-v_d x/2}\leq C x^{(2-\beta)(d+2)/(2d)}$ and 
$e^{-v_d x/2}\leq C x^{(2\ell-\beta)(d+2)/(2d)}$. Hence
\begin{align*}
\eta_N\indiq_{\Omega_N^3}\leq& C\frac{\bR_1}{N^{2/d}\bff_1^{2/d+1}} (\bh_{N1}\bff_1)^{(2-\beta)(d+2)/(2d)}
+C \Big(\frac{\bR_1}{N^{2/d}\bff_1^{2/d+1}}\Big)^\ell (\bh_{N1}\bff_1)^{(2\ell-\beta)(d+2)/(2d)}\\
=& C\frac{\bR_1^{\beta/2}}{N^{\beta/d}\bff_1^{\beta/d+\beta/2}}
\end{align*}
and  $\E[\eta_N\indiq_{\Omega_N^3}] \leq C N^{-\beta/d}$.

\vip

Finally on $\Omega_N^4$, we have  $\bh_{N1}\bff_1=\sqrt N \bff_1/ \bR_1$.
Moreover, there is a constant $C>0$ such that  $e^{-v_d x/2}\leq C x^{2(2-\beta)/d}$ and 
$e^{-v_d x/2}\leq C x^{2(2\ell-\beta)/d}$ for all $x>0$. Hence
\begin{align*}
\eta_N\indiq_{\Omega_N^4}\leq& C\frac{\bR_1}{N^{2/d}\bff_1^{2/d+1}} (\bh_{N1}\bff_1)^{2(2-\beta)/d}
+C \Big(\frac{\bR_1}{N^{2/d}\bff_1^{2/d+1}}\Big)^\ell (\bh_{N1}\bff_1)^{2(2\ell-\beta)/d}.
\end{align*}
We now use the Young inequality  with
$p=\beta/2$ and $p_*=\beta/(\beta-2)$ for the first term and 
$p=\beta/(2\ell)$ and $p_*=\beta/(\beta-2\ell)$ for the second one:
\begin{align*}
\eta_N\indiq_{\Omega_N^4}\leq& C\Big(\frac{\bR_1}{N^{2/d}\bff_1^{2/d+1}}\Big)^{\beta/2} 
+C (\bh_{N1}\bff_1)^{-2\beta/d}=\frac {C\bR_1^{\beta/2}}{N^{\beta/d}\bff_1^{\beta/d+\beta/2}}
+ \frac {C\bR_1^{2\beta/d}}{N^{\beta/d}\bff_1^{2\beta/d}}.
\end{align*}
Thus $\E[\eta_N\indiq_{\Omega_N^4}] \leq C N^{-\beta/d}$, which completes the proof.
\end{proof}

We quickly give the

\begin{proof}[Proof of Remark \ref{lambda1}]
Coming back to the proof of Lemma \ref{tedious}, we see
that $g_1(x,r)=-r^{2/d+1}\delta_2f(x)$, with $\delta_2f(x)=\sum_{|\alpha|=2}
(c_\alpha/\alpha!)\partial_\alpha f(x)$ and $c_\alpha=\int_{B(0,1)}y^\alpha \dd y$.
But for $\alpha=(\alpha_1,\dots,\alpha_d)$ with $|\alpha|=2$, we see that
$c_\alpha=0$ unless there is $i$ such that $\alpha_i=2$ and then
$c_\alpha>0$ does not depend on $i$ (because we work
with some symmetric norm).
Thus $g_1(x,r)=-c r^{2/d+1}\Delta f(x)$ for some constant $c>0$. 
Coming back to the proof of Lemma
\ref{lb5}, it holds that $\lambda_1=\E[\intoi e^{-v_d\bff_1}g_1(X_1,r)\frac{\dd r}{r}]$.
Thus, allowing the value of $c>0$ to vary,
$$
\lambda_1= - c \E\Big[\intoi e^{-v_d u} \Delta f(X_1)(u/\bff_1)^{2/d+1}\frac{\dd u}{u}\Big]
=-c\E\Big[\Delta f(X_1) /\bff_1^{2/d+1}\Big].
$$ 
As a consequence,
$$
\lambda_1= -c \intrd f^{-2/d}(x)\Delta f(x) \dd x=c \intrd f^{-2/d-1}(x)|\nabla f(x)|^2\dd x,
$$
the integration by parts being licit if $\lim_{|x|\to \infty}f^{-2/d}(x)|\nabla f(x)|=0$.
\end{proof}

We now have all the weapons to conclude the 

\begin{proof}[Proof of Theorem \ref{thb1}]
We fix $\beta\in (0,2]\cap(0,d]$, so that $\rho=\beta$.
We assume that $f\in \CC^\beta(\rd)$, we recall that
$M$ and $G_\beta$ were defined in \eqref{mM} and \eqref{gbeta}
and that $R=M+G_\beta$. We assume that $\kappa=\sup_{x\in \rd}f(B(x,1))<1$.
We assume \eqref{condb1} for some $\theta \in (0,\beta/d]$ and some $q>0$ (with $r_0=1$).
We recall that $\E[H_N]-H(f)= B^N_1+B^N_2+B^N_3+B^N_4$. 
It suffices to use Lemmas \ref{lb1}, \ref{lb2} and \ref{lb4} to find that 
$|B^N_1|+|B^N_2|+|B^N_3|+|B^N_4| \leq CN^{-\theta}$. The condition \eqref{condb1} indeed implies
that we can apply all these lemmas (this uses that $m\leq v_d f$).
\end{proof}

\begin{proof}[Proof of Theorem \ref{thb2}]
We fix $d\geq 3$ and $\beta\in (2,d]$, whence $\rho=2$ and
$\ell=\max\{i\in\nn\;:\;2i<\beta\}\geq 1$.
We assume that $f\in \CC^\beta(\rd)$, we recall that
$M$ and $G_\beta$ were defined in \eqref{mM} and \eqref{gbeta}
and that $R=M+G_\beta$. We assume that $\kappa=\sup_{\rd}f(B(x,1))<1$.
We assume \eqref{condb2} for some $q>0$ (with $r_0=1$).
We recall that $\E[H_N]-H(f)= B^N_1+B^N_2+B^N_3+B^N_4$. 
By Lemmas \ref{lb1} and \ref{lb2} with $\theta=\beta/d$, we have
$|B^N_2|+|B^N_3|+|B^N_4| \leq CN^{-\beta/d}$. Lemma \ref{lb5} tells us 
that $|B^N_1-\sum_{i=1}^\ell\lambda_iN^{-2i/d}|\leq CN^{-\beta/d}$.
We indeed can apply all these lemmas thanks to \eqref{condb2} (and since $m\leq v_d f$).
All this shows that $|\E[H_N]-H(f)-\sum_{i=1}^\ell\lambda_iN^{-2i/d}|\leq CN^{-\beta/d}$.
\end{proof}

\section{Corollaries and examples}\label{cc}

\subsection{Corollaries}

We start with a remark.

\begin{rk}\label{supercool}
(i) If $\intrd |x|^{d+\e}f(x)\dd x<\infty$ for some $\e\in(0,1)$, then 
$\intrd f^{1/2-\e'}(x)\dd x<\infty$, where $\e'=\e/(4d+2)$.
\vip

(ii) If $f$ is bounded, if $m \geq c f$ for some constant
$c>0$ and if
$\intrd |x|^{d+\e}f(x)\dd x<\infty$ for some $\e\in(0,1)$, then, with
$\theta=1/2+\e/(4d+2)$,
$$
\intrd \Big( \log^2 m(x)+\frac{\log^2(2+|x|)}{[f(B(x,1))]^{\theta}}
\Big)f(x)\dd x <\infty.
$$
\end{rk}

\begin{proof}
For point (i), we write $f^{1/2-\e'}(x)=f^{1/2-\e'}(x)(1+|x|)^{(d+\e)(1/2-\e')}(1+|x|)^{-(d+\e)(1/2-\e')}$ 
and we use the H\"older inequality with $p=1/(1/2-\e')$ and $p_*=1/(1/2+\e')$.
This gives $\intrd f^{1/2-\e'}(x)\dd x\leq I^{1/p}J^{1/p*}$,
with $I=\intrd f(x)(1+|x|)^{d+\e}\dd x<\infty$ by assumption and
$J=\intrd (1+|x|)^{-(d+\e)(1/2-\e')/(1/2+\e')}\dd x<\infty$ because
$(d+\e)(1/2-\e')/(1/2+\e')>d$.

\vip

For (ii), since $f \leq m/c \leq v_d f /c$ and since $f$ is bounded, we can find 
$C$ such that $\log^2 m \leq C(1+\log^2 f) \leq C f^{-\theta}$
and  $[f(B(x,1))]^{-\theta} \leq m^{-\theta}(x)\leq C f^{-\theta}(x)$. We thus only have to prove 
that $I=\intrd f^{1-\theta}(x)\log^2(2+|x|) \dd x <\infty$.
Since $1-\theta=1/2-\e'$, this is checked as point (i).
\end{proof}

We can now give the 

\begin{proof}[Proof of Corollary \ref{c1}]
We fix $\e\in(0,1)$ and assume that $f\in \CC^{\nu}(\rd)$ with $\nu=1$ if $d=1$, 
$\nu=2$ if $d\in\{2,3\}$ and $\nu=d/2+\e$ if $d\geq 4$.
We assume that $\kappa<1$ (with $r_0=1$), that $R=M+G_\nu$ is bounded and
that there is $c>0$ such that $m \geq cf$.
We finally assume that $\intrd |x|^{d+\e}f(x)\dd x<\infty$ and that 
(a) $\int_{\{f> 0\}} \sqrt {R(x)} f^{-\e}(x) \dd x <\infty$ if $d\in\{1,2\}$ or that
(b) $\int_{\{f> 0\}} R^{d/4}(x) f^{1/2-d/4-\e}(x)\dd x <\infty$ if $d\geq 3$.

\vip

{\it Step 1.} We can apply Theorem \ref{thv} and Proposition \ref{mp}
with $\theta= 1/2+\e/(4d+2)$.
Indeed, $f$ is continuous and bounded, so that we have only to check \eqref{condv} and 
\eqref{condestiv}. By Remark \eqref{supercool}-(ii), we only have to verify that
$\intrd (M(x)/m(x))(1+|\log m(x)|)f(x)\dd x<\infty$.
But since $(M/m)f\leq C M \leq C R$ and $|\log m| \leq C(1+|\log f|) \leq C f^{-\e}$
(because $m\leq C f \leq C$),
$(M/m)(1+|\log m|)f\leq C R f^{-\e}$.
If $d\in\{1,2\}$, we use that $R\leq C \sqrt R$ and conclude with (a). 
If $d\geq 3$, we write $R f^{-\e}\leq R f^{-\e} (R/f)^{d/4-1/2}
=R^{d/4+1/2} f^{1/2-d/4-\e}\leq C R^{d/4} f^{1/2-d/4-\e}$ and conclude with (b).

\vip

{\it Step 2.} We now show that we can apply Theorem \ref{thb1} with $\theta=1/2+\e/(4d+2)$ and $\beta=\nu$
when $d\in\{1,2,3\}$. By Remark \eqref{supercool}-(ii), 
we only have to verify that $I=\intrd (R(x)/m(x))^{2\theta}f(x)\dd x<\infty$
and $J=\intrd (R^{\theta d/\nu}(x)/m^{\theta(d+\nu)/\nu}(x) )f(x)\dd x<\infty$.
First, $(R/m)^{2\theta}f\leq C R^{1+\e/(2d+1)} f^{-\e/(2d+1)}$.
If $d\in\{1,2\}$, we write $(R/m)^{2\theta}f \leq  \sqrt Rf^{-\e}$,
so that $I<\infty$ by (a). If $d=3$, we write $(R/m)^{2\theta}f \leq C R^{3/4} f^{-1/4-\e}$ 
so that $I<\infty$ by (b).
Next, if $d\in\{1,2\}$, so that
$\nu=d$, we have $(R^{\theta d/\nu}/m^{\theta(d+\nu)/\nu})f\leq C R^{\theta}f^{1-2\theta}\leq
C \sqrt R f^{-\e}$, whence $J<\infty$ by (a). If $d=3$, then
$(R^{\theta d/\nu}/m^{\theta(d+\nu)/\nu})f\leq C R^{3\theta/2}f^{1-5\theta/2}\leq
CR^{3/4} f^{-1/4-\e}$, whence $J<\infty$ by (b).

\vip

{\it Step 3.} We now prove that we can apply Theorem \ref{thb2} with $\beta=d/2+\e d/(4d+2)
\in (d/2,\nu)$ when $d\geq 4$. 
Since $\beta/d =1/2+\e/(4d+2)$, we can use Remark \ref{supercool}-(ii)
and we only have to check that $I=\intrd (R(x)/m(x))^{2\beta/d}f(x)\dd x$ and 
$J= \intrd (R^{\beta/2}(x)/m^{\beta/d+\beta/2}(x) )f(x)\dd x$ are finite.
But $(R/m)^{2\beta/d}f\leq C Rf^{-\e/(2d+1)}\leq C Rf^{-\e}\leq C R f^{-\e} (R/f)^{d/4-1/2}
=CR^{d/4+1/2} f^{1/2-d/4-\e}\leq C R^{d/4} f^{1/2-d/4-\e}$. Next, 
$(R^{\beta/2}/m^{\beta/d+\beta/2})f\leq C R^{d/4}f^{1/2-d/4-\e(d+2)/(8d+4)}
\leq C R^{d/4}f^{1/2-d/4-\e}$. Hence $I$ and $J$ are finite by (b).

\vip

{\it Step 4.} Here we conclude when $d\in\{1,2,3\}$: by Step 1, we know that
$\sqrt N (H_N-\E[H_N]) \to \cN(0,\sigma^2(f))$ in law and that $V_N\to \sigma^2(f)$ in probability.
By Step 2, we know that $|\E[H_N]-H(f)|\leq C N^{-1/2-\e/(4d+2)}$, so that $\sqrt{N} (H_N-\E[H_N])\to 0$ 
in probability. We deduce that, as desired, $\sqrt{N/V_N} (H_N-H(f))\to \cN(0,1)$ in law.

\vip

{\it Step 5.} We now assume that $d\geq 4$ and observe that since $\e\in (0,1)$,
with $\beta=d/2+\e d/(4d+2)$, we have
$\ell=\max\{i\in\nn \; : \; 2i<\beta\}=\lfloor d/4\rfloor$.
By Step 1, we know that
$\sqrt N (H_N-\E[H_N]) \to \cN(0,\sigma^2(f))$ in law and that $V_N\to \sigma^2(f)$ in probability.
By Step 2, we know that there are some numbers $\lambda_1,\dots,\lambda_\ell$
so that $|\E[H_N]-H(f)-\sum_{i=1}^\ell\lambda_i N^{-2i/d}| \leq C N^{-1/2-\e/(4d+2)}$.

\vip

Recall now \eqref{richard}: we have $H_N^{(d)}=\sum_{k=0}^\ell \alpha_{k,d}H^{k}_{2^{\ell-k}n}$, where 
$n=\lfloor (N+1-\ell)/(2^{\ell+1}-1)\rfloor$, and where $H^0_{2^\ell n},\dots,H^\ell_{n}$ are independent.
For each $k=0,\dots,\ell$, we have 
\begin{gather}
\sqrt {2^{\ell-k}n} (H_{2^{\ell-k}n}^k-\E[H_{2^{\ell-k}n}^k]) \to \cN(0,\sigma^2(f)),\label{111}
\end{gather}
\begin{gather}
\Big|\E[H_{2^{\ell-k}n}^k]-H(f)-\sum_{i=1}^\ell\lambda_i (2^{\ell-k}n)^{-2i/d}\Big| 
\leq C (2^{\ell-k}n)^{-1/2-\e/(4d+2)}\leq C N^{-1/2-\e/(4d+2)}.\label{222}
\end{gather}
From \eqref{111} and since $n\sim (2^{\ell+1}-1)^{-1}N$, we conclude that
$\sqrt N (H_N^{(d)}-\E[H_N^{(d)}]) \to \cN(0,a_d \sigma^2(f))$, where
$a_d=(2-2^{-\ell})\sum_{k=0}^\ell \alpha_{k,d}^2 2^{k}$.
And the numbers $\alpha_{k,d}$ are such that
$\sum_{k=0}^\ell \alpha_{k,d} =1$ and $\sum_{k=0}^\ell \alpha_{k,d} 2^{2ki/d}=0$
for all $i=1,\dots,\ell$.
A little computation allows us to 
deduce from \eqref{222} that $|\E[H_N^{(d)}]-H(f)|\leq C N^{-1/2-\e/(4d+2)}$,
whence $\sqrt N |\E[H_N^{(d)}]-H(f)|\to 0$ in probability. All this proves that
$\sqrt N (H_N^{(d)}-H(f))\to \cN(0,a_d\sigma^2(f))$ in law.
Since finally $V_N\to \sigma^2(f)$ in probability, we conclude that 
$\sqrt {N/V_N} (H_N^{(d)}-H(f))\to \cN(0,a_d)$ in law as desired.
\end{proof}

We next give the

\begin{proof}[Proof of Corollary \ref{c2}]
We assume that $f\in \CC^\nu(\rd)$ with $\nu=\min\{d,2\}$, that $\kappa<1$ (with $r_0=1$),
that $R=M+G_\nu$ is bounded and
that there is $c>0$ such that $m\geq cf$. 
Assume finally that $\intrd |x|^{d+\e}f(x)\dd x<\infty$ for some $\e>0$, that
$\int_{\{f>0\}} M(x)|\log f(x)|\dd x<\infty$ and that
(a) $\intrd \sqrt{R(x)}  \dd x <\infty$ if $d=1$ and (b) $\intrd R^{d/(2+d)}(x)\dd x <\infty$
if $d\geq 2$. Let $\theta=1/2$ if $d=1$ and $\theta=2/(d+2)$ if $d\geq 2$. 

\vip

The only thing we have to verify is that we can apply Theorems \ref{thv} and \ref{thb1}
with this $\theta$ (and with $\beta=\nu$). We only have to verify
\eqref{condv} and \eqref{condb1}. By Remark \ref{supercool}-(ii) and since 
$(M/m)(1+|\log m|)f\leq C M+C M|\log f|$
(and $M$ is integrable because e.g. if $d\geq 2$, we have $M\leq R \leq C R^{d/(2+d)}$
and (b)), we conclude that \eqref{condv} holds with any $\theta \in (0,1/2+\e/(4d+2)]$
and thus in particular with our $\theta$.
For  \eqref{condb1}, we need $\intrd (R(x)/m(x))^{2\theta}f(x)\dd x$
and  $\intrd (R^{\theta d/\nu}(x)/m^{\theta (d+\nu)/\nu}(x))f(x)\dd x$ to be finite.

\vip

If $d=1$, we have  $(R/m)^{2\theta}f=(R/m)f\leq CR \leq C \sqrt R$, as well as 
$(R^{\theta d/\nu}/m^{\theta (d+\nu)/\nu})f=\sqrt R m^{-1}f\leq C\sqrt R$, whence the result by (a).

\vip

If $d\geq 2$, we have  $(R/m)^{2\theta}f(x)=(R/m)^{4/(d+2)}f \leq C R^{4/(d+2)}f^{1-4/(d+2)}\leq C R
\leq C R^{d/(2+d)}$ and $(R^{\theta d/\nu}/m^{\theta (d+\nu)/\nu})f
=R^{d/(d+2)}m^{-1}f\leq C R^{d/(d+2)}$, whence the result by (b).
\end{proof}

\subsection{Examples}
We finally verify that the examples of Subsection \ref{ex}
satisfy the announced properties.
It is always easily checked that there is 
$c>0$ such that $f\geq c m$, so that we omit the proof, except in example (f) where it is rather tedious.

\vip

\noindent {\bf (a)} If $f(x)=(2\pi)^{-d/2}\exp(-|x|^2/2)$,
Corollary \ref{c1} applies: we can take $r_0=1/2$ and the only difficulty is to check \eqref{cc2}.
But there is a constant $C$ such that
$R(x)\leq C \sup_{y\in B(x,1)} (1+|y|^{\lceil \nu\rceil})f(y) \leq C e^{2|x|}f(x)$.
Consequently, if $d\in\{1,2\}$, we have $\intrd \sqrt{R(x)}f^{-\e}(x)\dd x 
\leq C \intrd e^{|x|} f^{1/2-\e}(x)\dd x$ and
if $d\geq 3$, 
$\intrd R^{d/4}(x)f^{1/2-d/4-\e}(x)\dd x \leq C \intrd e^{d|x|/2} f^{1/2-\e}(x)\dd x$.
These integrals indeed converge if e.g. $\e=1/4$.

\vip

\noindent {\bf (b)} If $f(x)=c_{d,a} e^{-(1+|x|^2)^{a/2}}$ for some $a>0$, then 
Corollary \ref{c1} applies: we can take $r_0=1/2$ and the only difficulty is to check \eqref{cc2}.
If $a\in (0,1]$, there is a constant $C$
such that $R(x)\leq C \sup_{y\in B(x,1)} f(y) \leq C f(x)$ and \eqref{cc2} follows with
e.g. $\e=1/4$.
If now $a>1$, we have $R(x)\leq C \sup_{y\in B(x,1)} (1+|x|^{\lceil \nu\rceil (a-1)})f(y) 
\leq C e^{2a(1+|x|^2)^{(a-1)/2}} f(x)$ and, again,  \eqref{cc2} follows with $\e=1/4$.

\vip

\noindent {\bf (c)} If $f(x)=c_{d,a}(1+|x|^2)^{-(d+a)/2}$ with $a>d$, then
Corollary \ref{c1} applies: we can take $r_0=1/2$, we have $R\leq C f$ so that \eqref{cc2}
only requires (in any dimension) that $\intrd f^{1/2-\e}(x)\dd x<\infty$ and this is the case,
for $\e>0$ small enough, because $a>d$. Also, \eqref{cc1} with $\e>0$ small enough 
follows from the fact that $a>d$. 

\vip

\noindent {\bf (d)} If $f(x)=c_{d,a}|x|^{a}e^{-|x|}$ with $a>0$, then
$f$ belongs to $\CC^a(\rd)$, we can take $r_0=1/2$ and we have 
$M(x)+G_\beta(x)\leq C (1+|x|^{a})e^{-|x|}$ for $\beta\in(0,a]$.

\vip

\noindent $\bullet$ If $d=1$, $a\geq 1$ or $d=2$, $a\geq 2$, we can apply Corollary \ref{c1}:
\eqref{cc2} holds true for $\e\in(0,(d/a)\land (1/2))$.

\vip

\noindent $\bullet$ If $d=3$ and $a\in [2, 12)$, we can apply Corollary \ref{c1}:
\eqref{cc2} holds for $\e\in(0,1/2)$ such that $(1/4+\e)a<3$.

\vip

\noindent $\bullet$ If $d\in\{4,\dots,9\}$ and $a\in (d/2,4d/(d-2))$, 
we can apply Corollary \ref{c1}:
\eqref{cc2} holds true for $\e\in(0,1)$ such that $(d/4-1/2+\e)a<d$.

\vip

\noindent $\bullet$ If $d\geq 10$, we can never apply Corollary \ref{c1}.

\vip

\noindent $\bullet$ But, for any $d\geq 3$, $a\geq 2$, Corollary \ref{c2} applies.

\vip

\noindent {\bf (e)} If $f(x)=c_{a,b} \prod_{i=1}^d x_i^{a_i}(1-x_i)^{b_i}\indiq_{\{x_i\in [0,1]\}}$ for some
$a=(a_1,\dots,a_d)$ and $b=(b_1,\dots,b_d)$ both in $(0,\infty)^d$,
we set $\tau=\min\{a_1,b_1,\dots,a_d,b_d\}$ and $\mu=\max\{a_1,b_1,\dots,a_d,b_d\}$.
Then $f \in \CC^\tau(\rd)$ and we can take $r_0=1/4$. For
any $\beta \in (0,\tau]$, we have $M(x)+G_\beta(x)\leq C$ (so that for any
powers $\alpha>0$, $\eta>0$, $\int_{\{f>0\}} (M(x)+G_\beta(x))^\alpha f^{-\eta}(x)\dd x<\infty$
if and only if $\eta<1/\mu$).

\vip

\noindent $\bullet$ If $d=1$, $\tau\geq 1$ or $d=2$, $\tau\geq 2$, Corollary \ref{c1} applies:
\eqref{cc2} holds for $\e\in(0,1/\mu)$.

\vip

\noindent $\bullet$ If $d=3$, $2\leq \tau \leq \mu < 4$,  Corollary \ref{c1} applies:
\eqref{cc2} holds for $\e\in(0,1)$ such that $(1/4-\e)\mu<1$.

\vip

\noindent $\bullet$ If $d\geq 4$, we can never apply Corollary \ref{c1}. But for
any $d\geq 3$, $\tau\geq 2$, Corollary \ref{c2} applies.

\vip

\noindent {\bf (f)} If $d=1$ and $f(x)=c_p x^p|\sin(\pi/x)|\indiq_{\{x\in (0,1)\}}$ with $p\geq 2$,
then $f \in \CC^1(\rr)$ and we can choose $r_0=1/6$. To apply Corollary \ref{c1}, we need
to verify that (i) $\int_0^1 f^{-\e}(x)\dd x<\infty$ for some $\e>0$ and
(ii) there is $c>0$ such that $m\geq cf$.

\vip

\noindent We start with (i). 
By the Cauchy-Schwarz inequality and since $\int_0^1 x^{- p\e}(x)\dd x<\infty$ for $\e\in(0,1/p)$, 
it suffices to verify that $I=\int_0^1 |\sin(\pi/x)|^{-\e}\dd x<\infty$ for $\e>0$ small enough.
But $I=\int_1^\infty |\sin(\pi u)|^{-\e}\frac{\dd u}{u^2}\leq \sum_{k\geq 1} k^{-2}\int_k^{k+1}|\sin(\pi u)|^{-\e}\dd u$.
Now if  $\e\in(0,1)$, $\int_k^{k+1}|\sin(\pi u)|^{-\e}\dd u$ is finite and does not depend on $k$,
whence $I<\infty$ as desired.

\vip

\noindent For (ii), we need to verify that there is $c>0$ such that $F(x,r):=r^{-1}\int_{x-r}^{x+r}f(y)\dd y \geq c f(x)$
for all $x\in(0,1)$ and all $r\in(0,1/6)$. Let us give the main steps.

\vip

\noindent (A) For all $0\leq a \leq b\leq a+2$, $\int_a^b |\sin \pi u| \dd u
\geq c (b-a)(|\sin(\pi a)|+|\sin(\pi b)|)$. This relies on a little study: assume that
$0\leq a \leq b \leq 2$ by periodicity and separate the cases $0\leq a \leq b\leq 1$,
$0\leq a \leq 1 \leq b\leq 2$ and $1 \leq a \leq b\leq 2$.

\vip

\noindent (B) If $x\in [1/2,1)$, $F(x,r)\geq r^{-1}\int_{x-r}^{x}f(y)\dd y
\geq c  r^{-1}\int_{x-r}^{x} |\sin(\pi/y)| \dd y=c r^{-1} \int_a^b  |\sin(\pi u)| \frac{\dd u}{u^2}$,
with $a=x^{-1}$ and $b=(x-r)^{-1}$. Since $1\leq a \leq b \leq 3$, we deduce from (A) that
$F(x,r) \geq c r^{-1} \int_a^b  |\sin(\pi u)|\dd u \geq c r^{-1}  (b-a)|\sin(\pi a)|$.
Using again that $1\leq a \leq b \leq 3$, we see that $r^{-1}(b-a)=ab \geq 1$, so that finally,
$F(x,r) \geq c |\sin(\pi /x)|\geq c x^p |\sin(\pi /x)|$ as desired.

\vip

\noindent (C) 
If now $x\in (0,1/2)$, we write  $F(x,r)\geq c x^p r^{-1}\int_{x}^{x+r}|\sin(\pi/y)|\dd y=
c x^p r^{-1}\int_{a}^{b} |\sin(\pi u)|  \frac{\dd u}{u^2}$, with $a=(x+r)^{-1}$ and $b=x^{-1}$,
so that $3/2 \leq a \leq b$.

\vip

\noindent $\bullet$ If $\lfloor b \rfloor \in \{\lfloor a \rfloor,\lfloor a \rfloor+1\}$, which implies
that $b-a \leq 2$, we use point (A)
to write   $F(x,r)\geq c x^p r^{-1} b^{-2} (b-a)|\sin(\pi b)|\geq c x^p|\sin(\pi /x)|$ as desired.
We used that $r^{-1} b^{-2} (b-a)=a b^{-1} \geq 3/7$ because $3/2 \leq a \leq b\leq a+2$.

\vip

\noindent $\bullet$ If $\lfloor b \rfloor \geq \lfloor a \rfloor+2$, we write 
$$
F(x,r)\geq  c x^p r^{-1}
\sum_{k=\lfloor a \rfloor+1}^{\lfloor b \rfloor-1} k^{-2} \int_{k}^{k+1} |\sin(\pi u)|  \dd u
=c x^p r^{-1}\sum_{k=\lfloor a \rfloor+1}^{\lfloor b \rfloor-1} k^{-2}
\geq c x^p r^{-1}\int_{\lfloor a \rfloor+1}^{\lfloor b \rfloor} t^{-2}dt,
$$
whence $F(x,r)\geq  c x^p r^{-1}[(\lfloor a \rfloor+1)^{-1}-\lfloor b \rfloor^{-1}]$.
Using that $\lfloor b \rfloor \geq \lfloor a \rfloor+2\geq 3$, we easily conclude that
$F(x,r)\geq  c x^p r^{-1} (a^{-1}-b^{-1})=cx^p\geq c x^p|\sin(\pi /x)|$.

\end{document}